\documentclass[11pt]{imsart}
\RequirePackage{amsthm,amsmath,amsfonts,amssymb}
\RequirePackage[authoryear]{natbib}
\RequirePackage{graphicx}
\usepackage{mathtools}
\usepackage{xifthen}
\usepackage{bbm}
\usepackage{esvect}
\usepackage{amsmath} 
\usepackage{lipsum}  
\usepackage{pkgfile}
\usepackage[section]{placeins}
\startlocaldefs

\newtheorem{claim}{Claim}
\newtheorem{prop}{Proposition}
\newtheorem{rem}{Remark}
\newtheorem{defi}{Definition}
\newtheorem{theorem}{Theorem}
\newtheorem{notation}{Notation}

\renewcommand{\geq}{\geqslant}
\renewcommand{\leq}{\leqslant}

\renewcommand{\geq}{\geqslant}
\renewcommand{\leq}{\leqslant}

\newcommand{\dm}{{\mathrm{dim}}}

\newcommand{\p}{{\mathbb{P}}}
\newcommand{\e}{{\mathbb{E}}}

\newcommand{\bt}{{\bm t}}
\newcommand{\bh}{{\bm h}}
\newcommand{\bx}{{\bm x}}

\endlocaldefs

\begin{document}
	
	\begin{frontmatter}
		\title{Optimal Confidence Bands for Shape-restricted Regression in Multidimensions}
		\runtitle{Confidence Bands for Shape-restricted Regression}
		
		\begin{aug}
			\author[A.D.]{\fnms{Ashley (Pratyay)}~\snm{Datta}\ead[label=e1]{pd2511@columbia.edu}},
			\author[S.M.]{\fnms{Somabha}~\snm{Mukherjee}\ead[label=e2]{somabha@nus.edu.sg}\orcid{0000-0001-6772-3404}}
			\and
			\author[B.S.]{\fnms{Bodhisattva}~\snm{Sen}}\footnote{\footnotesize{Supported by NSF DMS-2311062.}\ead[label=e3]{bodhi@stat.columbia.edu}}
			\address[A.D.]{Department of Statistics,
				Columbia University\printead[presep={,\ }]{e1}}
			
			\address[S.M.]{Department of Statistics and Data Science, National University of Singapore\printead[presep={,\ }]{e2}}
			
			\address[B.S.]{Department of Statistics, Columbia University\printead[presep={,\ }]{e3}}
			
		\end{aug}
		
		\begin{abstract}
			In this paper, we propose and study construction of confidence bands for shape-constrained regression functions when the predictor is multivariate. In particular, we consider the continuous multidimensional white noise model given by $d Y(\bm t) = n^{1/2} f(\bm t) \,d\bm t + d W(\bm t)$, where $Y$ is the observed stochastic process on $[0,1]^d$ ($d\ge 1$), $W$ is the standard Brownian sheet on $[0,1]^d$, and $f$ is the unknown function of interest assumed to belong to a (shape-constrained) function class, e.g., coordinate-wise monotone functions or convex functions. The constructed confidence bands are based on local kernel averaging with bandwidth chosen automatically via a multivariate multiscale statistic. The confidence bands have guaranteed coverage for every $n$ and for every member of the underlying function class. Under monotonicity/convexity constraints on $f$, the proposed confidence bands automatically adapt (in terms of width) to the global and local (H\"{o}lder) smoothness and intrinsic dimensionality of the unknown $f$; the bands are also shown to be optimal in a certain sense. These bands have (almost) parametric ($n^{-1/2}$) widths when the underlying function has ``low-complexity'' (e.g., piecewise constant/affine).
		\end{abstract}
		
		
		
	\end{frontmatter}

	\section{Introduction}
	The area of shape-restricted regression is concerned with nonparametric estimation of a regression function  under shape constraints such as monotonicity, convexity, unimodality/quasiconvexity, etc. This field of statistics has a long history dating back to influential papers such as \cite{Hildreth-1954}, \cite{Brunk-1955}, \cite{Prakasa-Rao-1969}, \cite{Brunk-1970}, \cite{Groeneboom-Et-Al-2001}; also see~\cite{BBBB-1972},~\cite{Robertson-Et-Al-1988},~\cite{Groeneboom-Wellner-1992},~\cite{Groeneboom-Jongbloed-2014} for book length treatments on this topic. Indeed, such shape constraints arise naturally in various contexts: isotonic regression methods are widely employed in many real-life applications ranging from predicting ad click–through rates~\citep{Mcmahan-Et-Al-2013} to gene–gene interaction search~\citep{Luss-Et-Al-2012}; convex regression arises in 
	productivity analysis~\citep{allon2007nonparametric}, efficient frontier methods~\citep{Kuosmanen-Johnson-2010}, in stochastic control~\citep{Keshavarz-Et-Al-2011}, etc. Further, in many applications (such as estimation of production and utility functions in economics), justifying smoothness assumptions on the regression function is often impractical, whereas   qualitative assumptions like monotonicity and/or concavity are available; see e.g.,~\cite{chamb},~\cite{varian1992microeconomic},~\cite{allon2007nonparametric},~\cite{varianint},~\cite{chen2018shapeenforcing}.

	In recent years there has been much activity on  estimation of such shape-constrained regression functions with multivariate predictors; see e.g.,~\cite{Seijo-Sen-2011},~\cite{Han-Et-Al-2019},~\cite{Chen-Et-Al-2020},~\cite{Deng-Zhang-2020},~\cite{MukherjeePatra}. Further, in many recent papers the accuracy of the least squares estimator in these problems has been studied via finite sample (adaptive) risk bounds; see e.g.,~\cite{Meyer-Woodroofe-2000}, \cite{Zhang-2002}, \cite{Guntu-Sen-2015}, \cite{Chatterjee-Et-Al-2015}, \cite{Chatterjee-Et-Al-2018}, \cite{Guntu-Sen-2018}, \cite{Chatterjee-Lafferty-2019}, \cite{Han-Et-Al-2019}, \cite{Kur-Et-Al-2020}. However, the problem of inference with these multivariate shape-constrained functions, e.g., construction of confidence sets, is largely unexplored. Note that, in contrast, inference for univariate shape-constrained problems is well-studied with a large body of work in the last two decades, see e.g.,~\cite{Banerjee-Wellner-2001},~\cite{Sen-Boots-2010},~\cite{L_infty-Error-2012},~\cite{Divide-Conquer-2019},~\cite{Deng-Et-Al-2022} and the references therein.
	
	In this paper we consider construction of honest confidence bands for shape-constrained regression functions with multiple covariates with special emphasis to: (i) coordinate-wise nondecreasing (isotonic) functions and (ii) convex functions. Our proposed methodology, which extends the ideas of~\cite{Dumbgen-2003} to multiple dimensions, yields asymptotically optimal confidence bands that possess various adaptivity properties. In particular, for scenarios (i) and (ii) above, we prove spatial and local adaptivity of our confidence bands with respect to the (H\"{o}lder) smoothness of the underlying function and its intrinsic dimensionality. Our confidence bands are constructed using a multidimensional multiscale statistic. Note that in the recent literature many multidimensional multiscale statistics have been developed and studied (see e.g.,~\cite{Chan-Guenther-2013},~\cite{Konig-Et-Al-2020},~\cite{Arias-Castro2005},~\cite{Walther-Perry-2022},~\cite{Walther2010}); however in this paper we consider the proposal of~\cite{multiscale}---which is inspired by the one-dimensional multiscale statistic proposed and studied in~\cite{Dumbgen-Spokoiny-2001}---as it is most convenient for our setup.
	
	
	We consider the following continuous multidimensional white noise (regression) model: 
	\begin{equation}\label{eq:Mdl}
		{Y(\bm t)=\sqrt{n}\int_{0}^{t_1}\ldots\int_{0}^{t_d} f(s_1,\ldots,s_d)\; ds_d \ldots ds_1 + W(\bm t),}
	\end{equation}
	where $\bm t := (t_1,\ldots,t_d) \in [0,1]^d$, $d \ge 1$, $\{Y(t_1,\ldots,t_d): (t_1,\ldots,t_d)\in[0,1]^d \}$ is the observed data, $f \in L_1([0,1]^d)$ is the unknown (regression) function of interest, $W(\cdot)$ is the unobserved $d$-dimensional Brownian sheet (the multivariate generalization of the usual Brownian motion; see Definition~\ref{brsh} below), and $n$ is a known scale parameter. Estimation and inference in this model is closely related to that of (multivariate) nonparametric regression based on sample size $n$;  see e.g., \cite{BL96},~\cite{Mr01}.  We work with this white noise model as this formulation is more amiable to rescaling arguments; see e.g.,~\cite{Donoho1992}, \cite{Dumbgen-Spokoiny-2001}, \cite{Carter2006}.


	Given a function class $\mathcal{F} \subset L_1([0,1]^d)$ (e.g., $\mathcal{F}$ can be the class of coordinate-wise nondecreasing functions or the class of convex functions, defined on $[0,1]^d$), our goal is to construct an honest confidence band $[\hat{\ell},\hat{u}]$ for the true function $f \in \mathcal{F}$, i.e., find functions $\hat{\ell}$ and $\hat{u}$ depending on the observed data $Y(\cdot)$ (see~\eqref{eq:Mdl})  that satisfy:
	\begin{equation}\label{honconfd}
		\p_f \left(\hat{\ell}\le f\le \hat{u} \right) \ge 1-\alpha, \qquad \mbox{for all } \; f \in \mathcal{F}, \quad \mbox{for all } \; n \ge 1,
	\end{equation}
	for any given confidence level $\alpha \in (0,1)$. Here, by $\p_f(\cdot)$ we mean probability computed when the true function is $f$ in~\eqref{eq:Mdl}. Although nonparametric estimation of an unknown regression/density function based on smoothness assumptions using techniques such as kernels, splines and wavelets are abundant in the literature (see e.g.,~\cite{wahba},~\cite{Donoho-Johnstone-1994},~\cite{Wand-Jones-1995}, \cite{adaptingdonoho}, \cite{JFan92}, \cite{hart}, \cite{IJohn}, \cite{Gu-2002}, \cite{BrownCaiZhou}, \cite{projsp}), it is known that fully adaptive inference for certain smoothness function classes is not possible without making qualitative assumptions of some kind on the parameter space; see e.g.,~\citet[Theorem 8.3.11]{Gine-Nickl-2016} (also see~\cite{donoho},~\cite{Cai-Low-2004}). 
	
	Indeed, shape-constrained functions satisfy a \textit{two-sided bias inequality} (see~\eqref{lb7} below)---a crucial  assumption made in this paper for our proposed method---which enables the construction of uniform confidence bands with inferential guarantees like~\eqref{honconfd}.
	Non-asymptotic confidence bands under such shape constraints on the true function are available in the literature but only for one-dimensional function estimation problems (see e.g.,~\cite{davies},~\cite{Hengartner-Stark-1995},~\cite{Dumbgen-1998},~\cite{frn}).~\cite{Dumbgen-2003} derived asymptotically optimal confidence bands for the true regression function under shape constraints such as monotonicity and convexity in the continuous univariate white noise model, based on multiscale tests introduced in~\cite{Dumbgen-Spokoiny-2001}. 
	
	Generalizing the approach in \cite{Dumbgen-2003}, we construct a multiscale statistic in the multidimensional setting (also see \cite{multiscale}) which can be written as a supremum of local weighted averages of the response $Y(\cdot)$ with weights determined by a kernel function, parametrized by a vector of smoothing bandwidths and the centers of the kernel function (see \eqref{equT} below). These multiscale local averages are appropriately penalized to have them in the same footing. This ensures that the random fluctuations of the kernel estimator can be bounded uniformly in the bandwidth parameters, i.e., the supremum statistic remains finite almost surely (see \cite{multiscale}). Further, working with the supremum avoids the delicate choice of tuning parameters (smoothing bandwidths). 
	Using the definition of the supremum statistic $T$ (see~\eqref{equT}) and the two-sided bias condition (see \eqref{mainkercondition}), one can obtain pointwise lower and upper confidence bounds for the unknown function $f$ with guaranteed coverage (i.e., \eqref{honconfd} holds); see Theorem \ref{th1}.
	
	To intuitively understand the main ideas in our approach let us first look at how kernel averaging works in a neighborhood around a point. If we take a larger neighborhood, the bias of the kernel estimator increases whereas the variance of the estimator goes down. The optimal size of the neighborhood is the one that balances these two terms. Our method can be interpreted as kernel averaging over all possible scales (bandwidths) and then choosing the optimal neighborhood (see \eqref{lhat4} and \eqref{uhat4}). The multiscale statistic bounds the stochastic fluctuations (i.e., the variance term) uniformly over all scales and locations, whereas the bias condition~\eqref{mainkercondition} uniformly bounds the bias term. This results in confidence bands having guaranteed coverage.
	
	Our proposed confidence band automatically adapts to the underlying smoothness of the true  function $f$. To understand why, let us contrast the behavior of our band between the cases where the function is smooth versus when the function is rough. The bias term (obtained from kernel averaging) over a neighborhood is much larger when the function is rougher compared to when it is smooth, whereas the variance term does not change (with smoothness). If the function is smooth, our multiscale confidence band automatically chooses a much larger neighborhood to average over (which balances bias and variance) resulting in a shorter confidence band (with a faster rate of convergence).
	
	In particular, for coordinate-wise isotonic and multivariate convex functions, we show that our constructed confidence band
	is adaptive with respect to the H\"{o}lder smoothness of the underlying function (see Theorem \ref{Thm 2}) and the intrinsic dimensionality of the function, i.e., the number of variables/coordinates it truly depends on (see Theorem \ref{idimth}). The confidence band also exhibits local adaptivity (as shown in Theorems~\ref{spadth} and~\ref{locad}). To elaborate on the spatial adaptivity property, we show that, as a consequence of Theorem \ref{spadth}, if the true function is monotone and constant (or convex and affine) in an open neighborhood  of a point then our constructed confidence band achieves the parametric ($n^{-1/2}$) rate of convergence, uniformly on that neighborhood; see Remark~\ref{rem:Para-Simple}. This in particular, complements the near-parametric risk bounds developed for these ``low-complexity'' functions in recent years by various authors; see e.g.,~\cite{Chatterjee-Et-Al-2018},~\cite{Han-Zhang-2020},~\cite{Kur-Et-Al-2020},~\cite{Deng-Et-Al-2021}. 
	In Theorem \ref{locad}, we show that the width of our proposed confidence band at a point also adapts to the local H\"{o}lder smoothness of the underlying function. 

	
	Moreover, locally (at $\bm t_0\in (0,1)^d$) for a coordinate-wise  strictly increasing function, Theorem \ref{opconstant} shows that the width of our proposed confidence band (constructed using a special kernel function; see Section~\ref{sec:Choice-Fn-Class} and Theorem~\ref{thm2}) attains the minimax rate; moreover, it also essentially attains the minimax constant, which is given by the geometric mean of the gradient of the true function $f$ at $\bm t_0$. Analogously, Theorem \ref{opconstantcx} shows that we have a similar minimaxity property for multivariate convex functions, modulo slightly different constant factors. 


	The rest of the paper is organized as follows. In Section \ref{defnot}, we introduce some notation that will be necessary for stating the main results and analyses presented in this paper. In Section \ref{sec:Conf-Band}, we give the construction of our confidence band using our multiscale statistic; in Section~\ref{sec:Choice-Fn-Class} we discuss the choice of kernels necessary for this construction, under the natural shape constraints of monotonicity and convexity. Section \ref{sec:adapt} is devoted to proving several adaptivity properties of our constructed confidence band. In Section \ref{minimax}, we show that our proposed confidence band is optimal in a certain sense. In Section \ref{simulation studies} we illustrate our proposed confidence band via simulations for different choices of $f$; these bands (see e.g., Figures~\ref{Isotonic Confidence Bands-1}-\ref{Convex Confidence Bands-2}) illustrate the adaptive properties of our procedure.  The proofs of the main results are given in Section \ref{prmainthmbd}. Finally, proofs of some technical lemmas are provided in Appendix \ref{techres}.

	\section{Confidence bands for multivariate shape-restricted functions}
	
	\subsection{Preliminaries}\label{defnot}
	We now present some notation that we will use throughout the paper.
	
	\begin{notation}[Function classes $\mathcal{F}_1$ and $\mathcal{F}_2$]\label{Notation}
		We will denote by $\mathcal{F}_1$ the class of all coordinate-wise nondecreasing functions $f:[0,1]^d \to \R$, i.e.,~$f \in \mathcal{F}_1$ satisfy:
		$$f(x_1,\ldots, x_d) \leq f(y_1,\ldots, y_d)\qquad\text{if and only if} \quad x_i\leq y_i~\text{for all}~1\leq i \leq d,$$
		and by $\mathcal{F}_2$ the class of all convex functions $f:[0,1]^d \to \R$, i.e.,~$f \in \mathcal{F}_2$ satisfy:
		$$f(\lambda \bx + (1-\lambda)\bm y) \leq \lambda f(\bm x) + (1-\lambda) f(\bm y)\quad\text{for all}~\bm x, \bm y\in \mathbb{R}^d~\text{and}~\lambda \in [0,1].$$
	\end{notation}
	\noindent For measurable functions $g,h: \mathbb{R}^d \to \mathbb{R}$, we define:
	$$\langle g,h\rangle_{|B} := \int_{B} g(\bm x) h(\bm x)~d\bm x\quad\text{and}\quad \|g\|_{|B} := \sqrt{\langle g,g\rangle_B}~.$$ When $B := \mathbb{R}^d$, we just drop the subscripts in the above notation.
	For vectors $\bm a, \bm b \in \R^d$, we define:
	$$\bm a \star \bm b := (a_1b_1,\ldots,a_db_d).$$
	For ${\bm t} = (t_1,\ldots,t_d)  \in (0,1)^d$ and ${\bm h} =(h_1,\ldots,h_d) \in (0,1/2)^d$ let $$B_{\infty}({\bm t},{\bm h}):=\{{\bm x}=(x_1,\ldots,x_d) \in \R^d: |t_i-x_i|\leq h_i \; \; \forall i=1,\ldots,d\}.$$
	
	Next, we describe the Brownian sheet process (see~\eqref{eq:Mdl}), which will be important in our analysis.
	\begin{defi}[Brownian sheet]\label{brsh}
		By a $d$-dimensional Brownian sheet we mean a mean-zero Gaussian process $\{W(\bm t): \bm t\in [0,1]^d\}$ with covariance function given by:
		$$\mathrm{Cov}(W(\bm t),W(\bm s)) = \prod_{i=1}^d \min \{t_i,s_i\}, \qquad \mbox{for }\; \bm s =(s_1,\ldots, s_d), \; \bm t =(t_1,\ldots, t_d).$$

	\end{defi}
	Note that when $d=1$, $W(\cdot)$ reduces to the standard Brownian motion on [0,1]. We now mention some useful properties of the Brownian sheet process:
	\begin{itemize}
		\item If $g \in L_2([0,1]^d)$, then $\int g \,dW := \int_{[0,1]^d} g(\bm t) \,d W(\bm t) \sim N(0,\|g\|^2)$.
		
		\item If $g_1, g_2 \in L_2([0,1]^d)$, then $\mathrm{Cov}\left(\int g_1 \,dW, \int g_2 \,dW\right) = \int_{[0,1]^d} g_1(\bm t) g_2(\bm t) \,d\bm t$.
		
		\item \textit{Cameron-Martin-Girsanov theorem} (\cite{protter}): Let $\mu$ denote the measure induced by the Brownian sheet $W(\cdot)$ on the space $\Omega$ of all real-valued continuous functions on $[0,1]^d$, and let $\nu$ denote the measure induced by the process $Y$ defined in \eqref{eq:Mdl} on $\Omega$. Then $\nu$ is absolutely continuous with respect to $\mu$ and the Radon-Nikodym derivative is given by:
		$$\frac{d\nu}{d\mu} = \exp\left(\sqrt{n}\int f \, dW - \frac{n}{2} \|f\|^2\right).$$
	\end{itemize}

	For a more detailed study of the Brownian sheet see \cite{davar1}.

	\subsection{Proposed confidence band}\label{sec:Conf-Band}
	
	In this subsection we construct adaptive and optimal confidence bands $(\hat{\ell},\hat{u})$ for $f:[0,1]^d \to \R$ when $f$ is known to be shape-constrained, e.g., when $f$ is isotonic/convex, and our data is generated according to~\eqref{eq:Mdl}. Suppose that $\psi:\mathbb{R}^d \to \R$ is a measurable function of bounded Hardy-Krause variation (see~\citet[Definition A.1]{multiscale}) such that: 
	\begin{enumerate}
		\item[(i)] $\psi$ is $0$ outside $[-1,1]^d$;
		
		\item[(ii)]$\psi \in L^2(\R^d)$, i.e., $\int_{\R^d} \psi^2(\bm x) \,d\bm x <\infty$;
		
		\item[(iii)] $\int_{\R^d} \psi(\bm x) \,d\bm x > 0$.
	\end{enumerate}
	We call such a function a {\it kernel}. For any $\bm h:=(h_1,\ldots,h_d) \in I := (0,1/2]^d$ we define \begin{equation*}\label{ah} 
		A_{\bm h}:= \prod_{i=1}^d [h_i, 1-h_i] = \{{\bm t} = (t_1,\ldots, t_d) \in\mathbb{R}^d : h_i \leq t_i \leq 1-h_i \;\;\mbox{ for  } i =1,\ldots,d\}.
	\end{equation*} 
	For any $\bm t \in A_{\bm h}$ we define the centered (at $\bm t$) and scaled kernel function $\psi_{\bt,\bm h}:[0,1]^d \to \R$ as 
	{\begin{equation}\label{kerest}
			\psi_{\bt,\bm h}(\bm x) :=\psi \left(\frac{x_1-t_1}{h_1},\ldots,\frac{x_d-t_d}{h_d}\right), \qquad \mbox{for }\;\; {\bm x} = (x_1,\ldots, x_d) \in [0,1]^d.
	\end{equation} }
	For a fixed $\bm t \in A_{\bm h}$ we construct a kernel estimator $\hat{{f_{\bm h}}}(\bm t)$ of $f(\bm t)$ as: 
	\begin{equation}\label{kerequ-0}
		\hat{{f_{\bm h}}}(\bt):= \frac{1}{\sqrt{n} (\prod_{i=1}^d h_i)\langle 1,\psi\rangle} \int_{[0,1]^d} \psi_{\bt,\bm h}(\bm x) \, dY(\bm x).
	\end{equation}
	Elementary calculations show that \begin{eqnarray*} \E(\hat{f_{\bm h}}(\bt)) &=& \frac{\int_{[0,1]^d} \psi_{\bt,{\bm h}}(\bm x) f(\bm x)~d\bm x}{ (\prod_{i=1}^d h_i)\langle 1,\psi\rangle} = \frac{\langle f(\bm t + \bm h \star \cdot),\psi \rangle}{\langle 1,\psi\rangle} \\ 
		\Var(\hat{f_{\bm h}}(\bt)) &= & \frac{\|\psi_{\bt,\bm h}\|^2}{n(\prod_{i=1}^d h_i)^2\langle 1,\psi\rangle^2} = \frac{\|\psi\|^2}{n(\prod_{i=1}^d h_i)\langle 1,\psi\rangle^2 }.
	\end{eqnarray*}
	The main idea of our approach is to notice that the random fluctuations for these kernel estimators can be bounded uniformly in $\bm h$. To accomplish this, we look at the following multiscale statistic (with kernel $\psi$):  
	\begin{eqnarray}
		T(\pm \psi) & = & \sup_{\bm h \in I}~ \sup _{\bt \in A_{\bm h}} \left(\pm \frac{\int_{[0,1]^d} \psi_{\bt,\bm h}(\bm x) \, dW(\bm x) }{ (\prod_{i=1}^dh_i)^{1/2} \|{\psi}\|} - \Gamma(2^d\prod_{i=1}^dh_i)\right) \nonumber \\
		& = & \sup_{\bm h \in I} \sup _{\bt \in A_{\bh}} \left(\pm \frac{\hat{f_h}(\bt) - \E(\hat{f_\bh}(\bt))}{\Var^{1/2}(\hat{f_\bh}(\bt))}- \Gamma(2^d\prod_{i=1}^dh_i)\right) \label{equT}
	\end{eqnarray} 
	where $\Gamma(r) :=(2\log(e/r))^{1/2}$.~\citet[Theorem 2.1]{multiscale} guarantees that $T(\pm \psi)$ is finite almost surely.

	We assume that the unknown $f$ belongs to the function class $\mathcal{F}$ (which could be $\mathcal{F}_1$ or $\mathcal{F}_2$). In fact, the results in this section are valid for any function class that satisfies the following \textit{two-sided bias condition}: we assume that we can find kernels $\psi^\ell$ and $\psi^u$ such that the corresponding kernel estimators $\hat{f^{\ell}_{\bm h}}$ and $\hat{f^{u}_{\bm h}}$ (see~\eqref{kerequ-0}) satisfy:
	\begin{equation}\label{mainkercondition} 
		\E(\hat{f}^{\ell}_{\bm h}(\bt)) \leq f(\bm t) \leq \E(\hat{f}^u_{\bm h}(\bm t)) \qquad \mbox{for all } \bm h \in I, ~\bm t\in A_{\bm h} \mbox{ and } f\in \mathcal{F}.
	\end{equation} 
	We will show later that the above condition holds for the function classes $\mathcal{F}_1$ and $\mathcal{F}_2$ (see Section~\ref{sec:Choice-Fn-Class}); in fact, it holds for most shape-constrained function classes.
	
	In view of \eqref{mainkercondition} and the definition of $T$ (in~\eqref{equT}), we have the following for all $\bm h \in I$ and $\bm t \in A_{\bm h}$:
	\begin{eqnarray}\label{lb7}
		f(\bt) & = & \left\{f(\bt) -\E(\hat{f}^{\ell}_\bh(\bt)) \right\} + \left\{\E(\hat{f}^{\ell}_\bh(\bt)) - \hat{f}^{\ell}_\bh(\bt) \right\} + \hat{f}^{\ell}_\bh(\bt)\nonumber\\
		& \geq & \hat{f}^{\ell}_h(\bt) - \frac{\|\psi^\ell\| \left(T(\psi^\ell) + \Gamma(2^d \prod_{i=1}^d h_i)\right)}{\langle 1,\psi^\ell \rangle (n\prod_{i=1}^dh_i)^{1/2}} 
	\end{eqnarray} 
	and similarly,
	\begin{equation}\label{ub7}
		f(\bt)  \leq    \hat{f}^{u}_h(\bt) + \frac{\|\psi^u\| \left(T(-\psi^u) + \Gamma(2^d \prod_{i=1}^d h_i)\right)}{\langle 1,\psi^u \rangle (n\prod_{i=1}^dh_i)^{1/2}}. 
	\end{equation}
	Now, if $\kappa_\alpha$ denotes the $(1-\alpha)^{\mathrm{th}}$ quantile of the statistic:
	$$T^* := \max\{T(\psi^\ell), T(-\psi^u)\},$$
	then in view of \eqref{lb7} and \eqref{ub7}, we can define a $1-\alpha$ confidence band for $f$ as $[\hat{\ell},\hat{u}]$, where:
	\begin{eqnarray}
		\hat{\ell}(\bm t) := \sup_{\bm h \in I:~ \bm t \in A_{\bm h}} \left\{\hat{f}^{\ell}_h(\bt) - \frac{\|\psi^\ell\| \left(\kappa_\alpha + \Gamma(2^d \prod_{i=1}^d h_i)\right)}{\langle 1,\psi^\ell \rangle (n\prod_{i=1}^dh_i)^{1/2}}\right\}, \label{lhat4} \\
		\hat{u}(\bm t) := \inf_{\bm h \in I:~ \bm t \in A_{\bm h}} \left\{\hat{f^{u}_h}(\bt) + \frac{\|\psi^u\| \left(\kappa_\alpha + \Gamma(2^d \prod_{i=1}^d h_i)\right)}{\langle 1,\psi^u \rangle (n\prod_{i=1}^dh_i)^{1/2}}\right\}.\label{uhat4}
	\end{eqnarray}
	In view of \eqref{lb7} and \eqref{ub7}, we have:
	\begin{eqnarray}\label{guarantee1}
		\p_f\left(\hat{\ell}(\bm t) \leq f(\bm t) \leq \hat{u}(\bm t)~\text{for all}~\bm t\in [0,1]^d\right) &\ge& \p\left(T(\psi^\ell) \leq \kappa_\alpha~,~T(-\psi^u)\leq \kappa_\alpha\right)\nonumber\\ &=&\p\left(T^* \leq \kappa_\alpha\right)  = 1-\alpha.
	\end{eqnarray}
	
	This shows that $[\hat{\ell},\hat{u}]$ is indeed a confidence band with guaranteed coverage probability $1-\alpha$ for all $n\ge 1$; we state this formally below.
	
	\begin{theorem}\label{th1}
		For kernels $\psi^\ell$ and $\psi^u$ satisfying \eqref{mainkercondition}, we have:
		\begin{equation}\label{excvg}
			\p_f\left(\hat{\ell}(\bm t) \leq f(\bm t) \leq \hat{u}(\bm t)~\text{for all}~\bm t\in [0,1]^d\right) \ge 1-\alpha
		\end{equation}
		for all $f\in \mathcal{F}$ and for all $n\ge 1$.
	\end{theorem}
	
	The above theorem shows that for any function class $\mathcal{F}$ for which the two-sided bias bounds \eqref{mainkercondition} hold, our approach yields an honest finite sample confidence band for any $f\in \mathcal{F}$. It is natural to ask if the above constructed band is conservative in nature. In the following result (proved in Section~\ref{pf:exactcov}), we show that if for some function $f\in \mathcal{F}$, the function $-f$ also belongs to $\mathcal{F}$, then our confidence band has \textit{exact} coverage at $f$. 
	
	\begin{prop}\label{exactcov}
		Suppose $f,-f\in \mathcal{F}$ and \eqref{mainkercondition} holds. Then, our $1-\alpha$ confidence band $[\hat{\ell},\hat{u}]$ has exact coverage probability $1-\alpha$, i.e.,    $$\p_f\left(\hat{\ell}(\bm t) \leq f(\bm t) \leq \hat{u}(\bm t)~\text{for all}~\bm t\in [0,1]^d\right) = 1-\alpha.$$
	\end{prop}
	
	Proposition \ref{exactcov} shows that if $f\in \mathcal{F}_1$ is a constant or $f\in \mathcal{F}_2$ is an  affine function, then the coverage probability of our confidence band is exact.  We will now see that for certain functions that exhibit ``low-complexity" structure locally (for example, $f\in \mathcal{F}_1$ is locally constant or $f\in \mathcal{F}_2$ is locally affine), our confidence band exhibits adaptive rates, in particular it can shrink at the parametric $n^{-1/2}$ rate locally. The following result is proved in Section~\ref{prspadth}. 
	\begin{theorem}\label{spadth}
		Assume that~\eqref{mainkercondition} holds with kernels $\psi^\ell$ and $\psi^u$ for some function class $\mathcal{F}$. Suppose further that the true $f \in \mathcal{F}$ satisfies 
		\begin{equation}\label{eqty}
			\E(\hat{f}^{\ell}_{\varepsilon \boldsymbol{1}_d}(\bt)) = f(\bm t) = \E(\hat{f}^u_{\varepsilon \boldsymbol{1}_d}(\bm t)) \qquad \mbox{for all }\;\bm t \in D \subset  A_{\varepsilon \boldsymbol{1}_d},
		\end{equation}
		for some $\varepsilon > 0$.\footnote{Here, $\boldsymbol{1}_d$ refers to the $d$-dimensional vector with all entries $1$.} Then, 
		\begin{equation}\label{22117}
			\sup_{\bt\in D} ~(\hat{u}(\bt) - \hat{\ell}(\bt)) \le K_{\varepsilon} n^{-1/2}\left(|\kappa_\alpha| + |T(\psi^u)| + |T(-\psi^\ell)|\right)
		\end{equation}
		for some constant $K_{\varepsilon} > 0$ (depending on $\varepsilon, \psi^\ell,\psi^u$). If~\eqref{eqty} holds with $\varepsilon \equiv \varepsilon_n = (\log (en))^{-\frac{1}{d}}$ and for all $\bm t \in D_n \subset  A_{\varepsilon_n \boldsymbol{1}_d}$ then
		$$\sup_{\bt\in D_n} ~(\hat{u}(\bt) - \hat{\ell}(\bt)) \le K\rho_n\left(1+\frac{|\kappa_\alpha| + |T(\psi^u)| + |T(-\psi^\ell)|}{\sqrt{\log \log(en)}}
		\right)$$
		for some constant  $K> 0$ (depending on $\psi^\ell,\psi^u$), where $\rho_n := \left(\log(en)\log \log(e n)/n\right)^{1/2}$. 
	\end{theorem}

	To understand the implications of Theorem~\ref{spadth}, let us consider the special case where both $f$ and $-f$, restricted to a fixed subset $S$ of $[0,1]^d$, belongs to the class $\mathcal{F}$. In such a case, \eqref{eqty} holds for any small $\varepsilon>0$  and $D:=\{x \in (0,1)^d: B_{\infty}(x,\varepsilon \boldsymbol{1}_d) \subset S\}$, the `$\varepsilon$-interior' of $S$. In particular, the above result shows that if the true function $f \in \mathcal{F}_1$ is locally constant or if $f \in \mathcal{F}_2$ is locally affine in a fixed neighborhood $S$, then our confidence band automatically adapts to this structure, and shrinks at the parametric rate $n^{-1/2}$    on $D$. The second part of the result shows that a similar parametric rate $n^{-1/2}$ (up to multiplicative logarithmic factors) also holds if $D$ is allowed to grow to $S$ at a certain rate (with $n$).

	\begin{remark}[Confidence bands with (locally) parametric widths for ``low-complexity'' functions]\label{rem:Para-Simple}
		Suppose that $f \in \mathcal{F}_1$ is piecewise constant, or $f\in \mathcal{F}_2$ is piecewise affine. There is a large recent literature on how the least squares estimator in these shape-constrained problems adapt to this ``low-complexity'' structure of $f$ and exhibit (almost) parametric rates in terms of risk (see e.g,.~\cite{Guntu-Sen-2018},~\cite{Han-Et-Al-2019},~\cite{Kur-Et-Al-2020},~\cite{Deng-Et-Al-2021}). As an immediate consequence of Theorem~\ref{spadth} we see that our constructed confidence bands locally shrink at the parametric rate (see~\eqref{22117}) for such ``low-complexity'' functions.
	\end{remark}
	
	\subsection{Choice of kernels for function classes $\mathcal{F}_1$ and $\mathcal{F}_2$}\label{sec:Choice-Fn-Class}
	As we have mentioned in the Introduction, the two prime examples of shape-constrained function classes are: (i) the class of all $d$-dimensional coordinate-wise nondecreasing functions $\mathcal{F}_1$, and (ii) the class of all $d$-dimensional convex functions $\mathcal{F}_2$. In this subsection, we construct kernels $\psi^\ell$ and $\psi^u$ for each of these function classes $\mathcal{F}_1$ and $\mathcal{F}_2$, that satisfy the two-sided bias condition~\eqref{mainkercondition}. This would immediately imply that we can construct honest confidence bands for these function classes satisfying \eqref{excvg}.
	Moreover, the confidence bands constructed using these kernels will exhibit certain optimality properties (as will be shown in Theorem~\ref{opconstant} and \ref{opconstantcx} below).
	For the class $\mathcal{F}_1$ we define:
	\begin{equation}\label{psi1defnt}
		\psi_1^u(\bx) := \left(1-\sum_{i=1}^d x_i\right)\mathbbm{1}_{\bm x\in [0,\infty)^d,~\sum_{i=1}^d x_i\leq 1}\quad\text{and}\quad \psi_1^\ell(\bx) := \left(1+\sum_{i=1}^d x_i\right)\mathbbm{1}_{\bm x\in (-\infty,0]^d,~\sum_{i=1}^d x_i\ge -1}
	\end{equation}
	and for the class $\mathcal{F}_2$, we define:
	\begin{equation}\label{psi2defnt}
		\psi_2^u(\bx) := (1-\|\bx\|^2)\mathbbm{1}_{\|\bx\|\leq 1}\qquad\text{and}\qquad \psi_2^\ell(\bm x) := \left(1-\frac{2d+4}{d+1} \|\bx\| + \frac{d+3}{d+1}\|\bx\|^2\right)\mathbbm{1}_{\|\bx\|\leq 1}.
	\end{equation}
	Note that $\psi_2^\ell$ can take negative values.
	Theorem~\ref{thm2} below (proved in Section \ref{prmainthm}) shows that \eqref{mainkercondition} holds for these specific kernel choices. 
	\begin{theorem}\label{thm2}
		Let $\hat{f}_{\bh,j}^\ell$ and $\hat{f}_{\bh,j}^u$ denote the kernel estimators corresponding to the kernels $\psi_j^\ell$ and $\psi_j^u$ respectively, for $j\in \{1,2\}$. Then, \eqref{mainkercondition} holds for the function classes $\mathcal{F}_1$ and $\mathcal{F}_2$.
	\end{theorem}

	\section{Adaptivity of the confidence band}\label{sec:adapt}
	In this section we show that the width of our confidence band  $[\hat{\ell},\hat{u}]$ (see \eqref{lhat4} and \eqref{uhat4}) adapts to the (global and local) smoothness and the \textit{intrinsic dimension} of the true function $f$. Let us first define the rate of convergence for a confidence band as follows: We say that the confidence band {\it $\{(\ell(\bt),u(\bt)): \bt \in [0,1]^d\}$}, with coverage probability $1-\alpha$ (for $\alpha \in (0,1)$), has {\it rate of convergence} $\gamma_n$ on a set $A_n\subseteq [0,1]^d$ for a function class  $\mathcal{G}$  if  $$\inf_{{f \in \mathcal{G}}} \p_f\left(\sup_{\bt\in {A_n}} (u(\bt)-\ell(\bt)) \leq {\Delta \gamma_n}\right) \geq 1-\alpha, \quad\text{for all}~n,$$  
	where  ${\Delta}>0$ is a constant not depending on $n$ (but may depend on $\alpha$ and $\mathcal{G}$). Clearly we want the rate of convergence to be as small as possible.
	
	\subsection{Adaptivity with respect to global smoothness}\label{globalhold}
	
	To state our main result in this subsection we need to introduce the notion of H\"{o}lder smoothness of a function $f:[0,1]^d \to \R$, which we define below. 
	\begin{defi}[H\"{o}lder smoothness]\label{def:Holder}
		For every fixed $\beta>0$ and $L>0$, the H\"older class $\mathbb{H}_{\beta,L}$ on $[0,1]^d$ is defined as the set of all functions $f: [0,1]^d \to \R$ that have all partial derivatives of order $\lfloor \beta\rfloor$ (defined as the largest integer strictly less than $\beta$) on $[0,1]^d$, and satisfy:
		$$\sum_{\bm k \in \mathbb{N}^d: \|\bm k\|_1\leq \lfloor \beta\rfloor} \sup_{\bx \in [0,1]^d} \left|\frac{\partial^{\|\bm k\|_1} f(\bm x)}{\partial x_1^{k_1}\ldots \partial x_d^{k_d}}\right| \leq L$$
		and
		
		\begin{equation}\label{seccondhld}
			\sum_{\bm k \in \mathbb{N}^d: \|\bm k\|_1=\lfloor\beta\rfloor}\left|\frac{\partial^{\|\bm k\|_1} f(\bm y)}{\partial x_1^{k_1}\ldots \partial x_d^{k_d}}- \frac{\partial^{\|\bm k\|_1} f(\bm z)}{\partial x_1^{k_1}\ldots \partial x_d^{k_d}}\right| \leq L \|\bm y - \bm z\|^{\beta-\lfloor \beta\rfloor}\quad \text{for all}~\bm y,\bm z \in [0,1]^d~.  
		\end{equation}
		
	\end{defi}

	The following theorem shows that the rate of convergence of our confidence band $[\hat{\ell},\hat{u}]$ for the class $\mathbb{H}_{\beta,L}\cap \mathcal{F}_j$ (for $j=1,2$) is $(\log n/n)^{\beta/(2\beta+d)}$. Note that the construction of the confidence band does not use knowledge of the smoothness parameter $\beta$ and it still achieves the minimax rate of convergence for the class $\mathbb{H}_{\beta,L}$ (see~\cite{Tsybakov-2009}) if $f \in \mathbb{H}_{\beta,L}$. This highlights the adaptive nature of our confidence band.
	\begin{theorem}\label{Thm 2}
		Fix $j\in \{1,2\}$, $f \in \mathcal{F}_j \cap \mathbb{H}_{\beta,L}$ with $j-1 <\beta \leq j$ and $L>0$. Suppose that  $[\hat{\ell},\hat{u}]$ is the level $\alpha$ confidence band constructed as in~\eqref{lhat4} and~\eqref{uhat4} based on kernel functions $\psi^\ell,\psi^u$ satisfying the two-sided bias condition~\eqref{mainkercondition}. Then there exists a constant $\Delta>0$ (depending only on $L,\beta,\psi^\ell,\psi^u$) such that $$\sup_{\bt \in A_{\varepsilon_n \boldsymbol{1}_d}}\left(\hat{u}(\bt)-\hat{\ell}(\bt)\right) \leq \Delta\, \varepsilon_n^\beta\left(1+\frac{|\kappa_\alpha|+ |T(-\psi^\ell)|+|T(\psi^u)|}{(\log(en))^{1/2}}\right)$$ where $\varepsilon_n :=(\log(en)/n)^{1/(2\beta+d)}$, and $\boldsymbol{1}_d$ is the $d$-dimensional vector of all ones. This, in particular, implies that 
		$$\inf_{f\in\mathcal{F}_j \cap \mathbb{H}_{\beta,L}}\p\left(\sup_{\bt \in A_{\varepsilon_n \boldsymbol{1}_d}}(\hat{u}(\bt)-\hat{\ell}(\bt)) \leq \Delta \Big[\frac{\log(en)}{n}\Big]^{\frac{\beta}{2\beta+d}}\right) \geq 1-\alpha, \quad \mbox{for all } n, $$
		for some constant $\Delta>0$ depending only on $L,\beta,\psi^\ell,\psi^u,\alpha$. 
	\end{theorem}
	
	Theorem \ref{Thm 2} is proved in Section \ref{thm2prf}.
	Its proof starts by bounding the pointwise deviation of the upper (and lower) band of our constructed confidence set from the true function $f$, in terms of the inner product of the variation of $f$ in a small neighborhood. The variation of $f$ over this neighborhood can then be bounded in terms of appropriate powers of the smoothing bandwidth, using the H\"{o}lder smoothness of $f$. 
	
	\subsection{Adaptivity with respect to local smoothness}
	Our main result in Section \ref{globalhold} showed that our confidence band achieves the optimal rate of convergence when the true function is globally H\"older smooth. In this subsection we show that our adaptivity results also hold when the true function is locally H\"older smooth. Specifically, we will look at the behavior of $\hat{u}(\bm t_0) - \hat{\ell}(\bm t_0)$ for a fixed $\bm t_0\in (0,1)^d$.    
	
	\begin{theorem}\label{locad}
		Fix $j \in \{1,2\}$ and ${\bm t}_0 \in (0,1)^d$. Suppose that $f \in \mathcal{F}_j$, and that there exists $\varepsilon \in (0,1/2) $ such that $f$ is H\"{o}lder smooth on ${B}_{\infty}(\bm t_0,\varepsilon)$ with smoothness parameter $\beta \in (j-1,j]$ and $L>0$. We construct the confidence band $[\hat{\ell},\hat{u}]$ based on kernel functions $\psi^\ell,\psi^u$ that satisfy the two-sided bias condition~\eqref{mainkercondition}. Then there exists a constant $K$ depending on $\beta, L,\psi^\ell,\psi^u$ such that:
		\begin{equation*}\label{locad1}
			\hat{u}(\bt_0) - \hat{\ell}(\bt_0) \le K\rho_n \left(1+\frac{|\kappa_\alpha| + |T(-\psi^\ell)| + |T(\psi^u)|}{(\log(en))^{1/2}}\right)
		\end{equation*}
		where $\rho_n = (\log(en)/n)^{\beta/(2\beta+d)}$. Note that this implies that
		$$\p_f\left(\hat{u}(\bt_0)-\hat{\ell}(\bt_0) \leq \Delta \Big[\frac{\log(en)}{n}\Big]^{\frac{\beta}{2\beta+d}}\right) \geq 1-\alpha, \quad \mbox{for all } n, $$
		for some constant $\Delta>0$ depending only on $L,\beta,\psi^l,\psi^u,\alpha$.
	\end{theorem}
	
	The proof of Theorem \ref{locad} essentially follows from that of Theorem~\ref{Thm 2} by noticing that one needs to control the bias of the kernel estimators locally which in turn only depends on the local smoothness of $f$; see Section~\ref{sec:Loc-Smooth} for a sketch of the proof.



	\subsection{Adaptivity with respect to intrinsic dimension}
	The intrinsic dimension of a function refers to the number of coordinates it actually depends on.
	\begin{defi}[Intrinsic dimensionality]\label{iddef}
		The intrinsic dimension $\dm(f)$ of a function $f: \R^d \to \R$ is $k\ge 1$ if and only if:
		\begin{enumerate}
			\item there exist $1\leq i_1<i_2<\ldots<i_k\leq d$ and a function $g: \R^k \to \R$ such that $f(x_1,\ldots,x_d) = g(x_{i_1},\ldots,x_{i_k})$ for all $(x_1,\ldots,x_d) \in \R^d$, and
			
			\item $f$ is not a function of $(x_s)_{s\in S}$ for any strict subset $S$ of $\{i_1,\ldots,i_k\}$.
		\end{enumerate}
	\end{defi}
	It can be verified easily from Definition \ref{iddef} that the intrinsic dimensionality of a function $f$ is unique. We will now show that our confidence band $[\hat{\ell},\hat{u}]$ adapts to the intrinsic dimensionality of the true function $f$.
	
	\begin{theorem}\label{idimth}
		Fix $j\in \{1,2\}$, $f \in \mathcal{F}_j \cap \mathbb{H}_{\beta,L}$ with $j-1 <\beta \leq j$ and $L>0$. Let $\dm(f) = k$ and suppose that $f(\bm x) = g(x_{i_1},\ldots,x_{i_k})$ for all $\bm x=(x_1,\ldots,x_d)\in [0,1]^d$ and some $1\le i_1<\ldots < i_k\le d$ and some function $g:[0,1]^k \to \R$. Suppose that $[\hat{\ell},\hat{u}]$ is the constructed confidence band based on kernel functions $\psi^\ell,\psi^u$ that satisfies the two-sided bias condition \eqref{mainkercondition}. Then, for every $\varepsilon \in (0,1/2)$, there exists a constant $\Delta>0$ (depending only on $L,\beta,\psi^\ell,\psi^u,\varepsilon$) such that:
		$$\sup_{\bm t \in A_{\boldsymbol{\varepsilon}_{n,i_1,\ldots,i_k}}}\left(\hat{u}(\bm t)-\hat{\ell}(\bm t)\right) \le \Delta \rho_{n,k} \left(1+\frac{|\kappa_{\alpha}| + |T(-\psi^\ell)| + |T(\psi^u)|}{(\log(en))^{1/2}}\right)$$
		where $\rho_{n,k} := (\log(en)/n)^{\beta/(2\beta+k)}$, $\varepsilon_n := \varepsilon \rho_{n,k}^{1/\beta}$ and $\boldsymbol{\varepsilon}_{n,i_1,\ldots,i_k}$ is the $d$-dimensional vector with the $i_1^{\mathrm{th}},\ldots,i_k^{\mathrm{th}}$ entries all equal to $\varepsilon_n$ and all other entries equal to $\varepsilon$.
	\end{theorem}
	
	Theorem \ref{idimth}, proved in Section \ref{pridim}, states that the rate of convergence of our confidence band when the function $f$ lies in $\mathcal{F}_j \cap \mathbb{H}_{\beta,L}$ (for $j = 1,2$) with $j-1 <\beta \leq j$ and $L>0$, is $\left(\log(en)/n\right)^{\beta/(2\beta+\dm(f))}$. Thus, the rate of convergence of the confidence band depends only on the variables which actually affect the true function, and not on the redundant variables that the function does not vary with. This is a desirable property, as this shows that our proposed confidence band can avoid the curse of dimensionality when $d$ is large, if dim($f$) is small. Note that this is another example of automatic adaptation of our procedure: the proposed confidence band has no knowledge of dim($f$) and it still automatically adjusts to yield a band that shrinks at the correct rate.

	\section{Optimality of our confidence band}\label{minimax}
	In this section we prove that our proposed confidence band (see \eqref{lhat4} and~\eqref{uhat4}) is optimal in a certain sense. Our results extend Theorem 4.2 in \cite{Dumbgen-2003}. In order to state our result, we need some notation. For a function $g: \mathbb{R}^d\to \mathbb{R}$ and $U\subseteq \mathbb{R}^d$, define:
	$$\|g\|_U := \sup_{\bm x\in U} |g(\bm x)|.$$ We first state our optimality result for the class of coordinate-wise nondecreasing functions $\mathcal{F}_1$.
	
	\begin{theorem}\label{opconstant}
		Let $f \in \mathcal{F}_1$ be a continuously differentiable function in an open neighborhood $U$ of $\bm t_0 \in (0,1)^d$ such that 
		\begin{equation*}
			L_1 \equiv L_1[\bm t_0] \equiv L_1[f,\bm t_0]:= \left[\prod_{i=1}^d \frac{\partial}{\partial x_i} f(x) \big|_{x=\bm t_0}\right]^{1/d} > 0.
		\end{equation*}
		Define
		\begin{equation}\label{eq:Delta_z}
			\rho_n :=\left(\frac{\log(en)}{n} \right)^{\frac{1}{2+d}} \quad \text{and } \quad \Delta^{(z)}:=\left(\left(\frac{d+2}{2d}\right) \|{\psi_1^{z}}\|^2\right)^{-\frac{1}{2+d}}
		\end{equation}
		where $z$ stands for $u$ and $\ell$ corresponding to kernels $\psi_1^u$ and $\psi_1^\ell$ (respectively) as defined in \eqref{psi1defnt}. Then we have the following:
		\begin{itemize}
			\item [(a)]
			Let $[\ell,u]$  be any confidence band such that, for some $\alpha \in (0,1)$, $$\p_g\big(\ell(\bm t) \leq g(\bm t) \leq u(\bm t) \mbox{ for all } \bm t \in [0,1]^d\big) \geq 1-\alpha  \quad \mbox{ for all } g \in \mathcal{F}_1.$$ Then, for any $\epsilon>0$,
			\begin{equation*}
				\begin{split}
					\liminf_{n \to \infty }\p_{f}\left(\|f- \ell\|_{U} \geq (1-\epsilon) \Delta^{(\ell)} L_1^{\frac{d}{2+d}}[f,\bm t_0] \rho_n\right) \geq 1-\alpha, \\
					\liminf_{n \to \infty} \p_{f}\left(\|u-f\|_{U} \geq (1-\epsilon) \Delta^{(u)} L_1^{\frac{d}{2+d}}[f,\bm t_0] \rho_n\right) \geq 1-\alpha.
				\end{split}
			\end{equation*}

			\item[(b)] Moreover, let $[\hat{\ell},\hat{u}]$ be the confidence band with coverage probability $1-\alpha$ as defined in \eqref{lhat4} and \eqref{uhat4}, with kernels as in \eqref{psi1defnt}. Then, for any $\epsilon>0$, we have 
			\begin{equation}\label{monotoneupp}
				\begin{split}
					\lim_{n \to \infty} \p_{f}\left((f-\hat{\ell})(\bm t_0) \leq (1+\epsilon)\Delta^{(\ell)}L_1^{\frac{d}{2+d}}[f,\bm t_0] \rho_n\right) = 1,\\
					\lim_{n \to \infty} \p_{f}\left((\hat{u}-f)(\bm t_0) \leq (1+\epsilon)\Delta^{(u)} L_1^{\frac{d}{2+d}}[f,\bm t_0] \rho_n\right) = 1.
				\end{split}
			\end{equation}
		\end{itemize}
	\end{theorem}
	Theorem \ref{opconstant}  (proved in Section \ref{optl}) states that the length of any confidence band for $f\in \mathcal{F}_1$ with guaranteed coverage probability $1-\alpha$, is at least $(\log n/n)^{1/(2+d)}$ up to a constant factor. Further, this optimal length is essentially achieved locally at ${\bm t}_0$ by our constructed confidence band. We can, in particular, exactly compute the constants $\Delta^{(\ell)}$ and $\Delta^{(u)}$ appearing in the above result (see~\eqref{eq:Delta_z}). For example, for $d=2$, $\Delta^{(\ell)} = \Delta^{(u)}  \approx 1.86121$.

	Note that the asymptotic probabilities in the upper bound results in part (b) of Theorem \ref{opconstant} do not depend on $\alpha$, unlike the corresponding lower bound probabilities in part (a). This can be understood from the fact that the random variable in part (a) is a supremum of pointwise deviations of the lower and upper confidence bands from the true function over a neighborhood, unlike the corresponding random variable in part (b), which just captures the deviation at $\bm t_0$. To draw a simple analogy, note that the $\alpha^{\mathrm{th}}$ quantile $q_\alpha^{(n)}$ of the maximum of a sequence of i.i.d.~Gaussians $Z_1,\ldots,Z_n$ satisfies $\p(\max_{1\le i\le n} Z_i\ge q_\alpha^{(n)}) = 1-\alpha$, but since $q_\alpha^{(n)}$ is of the order $\sqrt{\log n}$, $\p(Z_1\le q_\alpha^{(n)}) = 1-o(1)$.\footnote{Another analogy where the quantiles of the supremum have the same order as each individual random variable is given below: Consider the supremum of the Brownian motion on $[0,1]$ which has the same distribution as the absolute value of a standard Gaussian. In this case, the quantiles of the supremum is of the same order as the corresponding quantiles of the process at any single point, but the former is always larger than the latter.}

	Next, we state our optimality result for the class of multivariate convex functions $\mathcal{F}_2$. Theorem~\ref{opconstantcx} below is proved in Section~\ref{propconstantcx}.

	{\color{black}
		\begin{theorem}\label{opconstantcx}
			Let $f \in \mathcal{F}_2$ be a twice continuously differentiable function in an open neighborhood $U$ of $\bm t_0 \in (0,1)^d$ such that 
			\begin{equation*}
				L_2 \equiv L_2[\bm t_0] \equiv L_2[f,\bm t_0]:= \mathrm{det}(H(\bm t_0))^{1/d} > 0
			\end{equation*}
			where $H(\bm t_0) \in \mathbb{R}^{d\times d}$ denotes the Hessian of $f$ at $\bm t_0$, and $\mathrm{det}(\cdot)$ denotes the determinant operator. Define $$\rho_n :=\left(\frac{\log(en)}{n} \right)^{\frac{2}{4+d}}~,$$  $$\Delta^{(\ell)}:=\left(\frac{d+4}{2d}\left(\sqrt{\frac{2(d+3)}{d+1}}\right)^d \|{\psi^{\ell}_2}\|^2\right)^{-\frac{2}{4+d}},\qquad\qquad \Delta^{(\ell\star)} := \Delta^{(\ell)} d^\frac{d}{d+4}$$
			$$\Delta^{(u)}:=\left(\frac{d+4}{2d}\left(\sqrt{2}\right)^d \|{\psi^{u}_2}\|^2\right)^{-\frac{2}{4+d}}\qquad\text{and} \qquad \Delta^{(u\star)} := \Delta^{(u)} d^\frac{d}{d+4},$$
			where $\psi_2^u$ and $\psi_2^\ell$ (respectively) as defined in \eqref{psi2defnt}.
			Also, define
			$$L_{2,\star} \equiv L_{2,\star}[\bm t_0] \equiv L_{2,\star}[f,\bm t_0] := \left[\prod_{i=1}^d H(\bm t_0)_{i,i}\right]^{1/d}.$$ 
			Then we have the following:
			\begin{itemize}
				\item [(a)]
				Let $[\ell,u]$  be any confidence band such that, for some $\alpha \in (0,1)$, $$\p_g\big(\ell(\bm t) \leq g(\bm t) \leq u(\bm t) \mbox{ for all } \bm t \in [0,1]^d\big) \geq 1-\alpha  \quad \mbox{ for all } g \in \mathcal{F}_2.$$ Then, for any $\epsilon>0$,
				\begin{equation}\label{optcond1}
					\begin{split}
						\liminf_{n \to \infty }\p_{f}\left(\|f- \ell\|_{U} \geq (1-\epsilon) \Delta^{(\ell)} L_2^{\frac{d}{4+d}}[f,\bm t_0] \rho_n\right) \geq 1-\alpha, \\
						\liminf_{n \to \infty} \p_{f}\left(\|u-f\|_{U} \geq (1-\epsilon) \Delta^{(u)} L_2^{\frac{d}{4+d}}[f,\bm t_0] \rho_n\right) \geq 1-\alpha.
					\end{split}
				\end{equation}

				\item[(b)] Moreover, let $[\hat{\ell},\hat{u}]$ be the confidence band with coverage probability $1-\alpha$ as defined in \eqref{lhat4} and \eqref{uhat4}, with kernels as in \eqref{psi2defnt}. Then, for any $\epsilon>0$, we have 
				\begin{equation}\label{optcond22}
					\begin{split}
						\lim_{n \to \infty} \p_{f}\left((f-\hat{\ell})(\bm t_0) \leq (1+\epsilon)\Delta^{(\ell \star)}L_{2,\star}^{\frac{d}{4+d}}[f,\bm t_0] \rho_n\right) = 1,\\
						\lim_{n \to \infty} \p_{f}\left((\hat{u}-f)(\bm t_0) \leq (1+\epsilon)\Delta^{(u \star)} L_{2,\star}^{\frac{d}{4+d}}[f,\bm t_0] \rho_n\right) = 1.
					\end{split}
				\end{equation}
				
				\item[(c)] In the special case when the Hessian matrix of $f$ at $\bm t_0$ is diagonal, one has the following for every $\epsilon > 0$,
				\begin{equation}\label{optcond223}
					\begin{split}
						\lim_{n \to \infty} \p_{f}\left((f-\hat{\ell})(\bm t_0) \leq (1+\epsilon)\Delta^{(\ell)}L_{2}^{\frac{d}{4+d}}[f,\bm t_0] \rho_n\right) = 1,\\
						\lim_{n \to \infty} \p_{f}\left((\hat{u}-f)(\bm t_0) \leq (1+\epsilon)\Delta^{(u)} L_{2}^{\frac{d}{4+d}}[f,\bm t_0] \rho_n\right) = 1.
					\end{split}
				\end{equation}

			\end{itemize}
		\end{theorem}
		
	}
	Part (a) of the above result gives a lower bound on the maximal (local) deviation of any honest confidence band (for the class of convex functions) around the true function.  We can, in particular, exactly compute the constants $\Delta^{(\ell)}$ and $\Delta^{(u)}$ appearing in the above result. For example, for $d=2$, $\Delta^{(\ell)} \approx 1.19788$ and $\Delta^{(u)} \approx 0.68278$. Part (c) shows that our confidence band attains this minimal length at $\bm t_0$ when the true function $f$ has a diagonal Hessian matrix at $\bm t_0$. In contrast, if the Hessian is not diagonal, our confidence band may have slightly larger width at ${\bm t}_0$, as indicated in part (b). The following remark discusses this discrepancy.
	
	\begin{rem}[On the lower and upper bounds in Theorem \ref{opconstantcx}]
		{ We would like to highlight that the constant terms in the upper and lower bounds appearing  in \eqref{optcond1} and \eqref{optcond22} are different. This discrepancy is reflected by the replacement of the geometric mean of the spectrum of the Hessian of $f$ at $\bm t_0$ (i.e., $L_2[f,\bm t_0]$) by the geometric mean of the diagonal entries of the Hessian matrix (i.e., $L_{2,\star}[f,\bm t_0]$). Further, there is an inflation by a multiplicative factor depending on the dimension $d$ (see Theorem \ref{opconstantcx} part (b)). The main two terms $L_2$ and $L_{2,\star}$ in the lower and upper bounds in \eqref{optcond1} and \eqref{optcond22} satisfy $L_2 \le L_{2,\star}$. }  Intuitively, this discrepancy is caused because we are limited to choosing a diagonal bandwidth matrix, whereas the orientation of the underlying multivariate convex function $f$ is not necessarily restricted along the standard coordinate axes. Note that when $d=1$, $L_{2} = L_{2,\star}$ and $\Delta^{(z)} = \Delta^{(z\star)}$ for $z\in \{u,\ell\}$, and hence both the upper and lower bounds match (as in \cite{Dumbgen-2003}). Observe that such a discrepancy does not happen for coordinate-wise monotone functions (cf.~Theorem \ref{opconstant}), as these functions are inherently tied to the standard coordinate axes.
	\end{rem}

	\begin{rem}[On the proofs of part (a) of Theorems \ref{opconstant} and \ref{opconstantcx}]
		The proofs of the lower bounds (part (a) of Theorems~\ref{opconstant} and \ref{opconstantcx}) involve the following main ideas. As a first step, one constructs a grid in $[0,1]^d$ with spacings given by a bandwidth parameter and centered at $\bm t_0$. For each such grid point $\bm t$, one defines a function $f_{\bm t}$ by perturbing the true function $f$ suitably. 
		The second step is to show that all these perturbed functions satisfy the corresponding shape constraint (see Lemmas~\ref{ftinF1} and~\ref{ftinF2}). As a next step, one shows that the probability that the deviation between the upper and lower limits (of any honest confidence band with coverage probability $1-\alpha$) exceeds a suitable constant (depending on the bandwidth parameters) is lower bounded by $1-\alpha$ minus a remainder term (depending on the perturbation functions). 
		One then argues that this remainder term is asymptotically negligible by expressing it in terms of an average of the likelihood ratio between the measures at the perturbed and the true function, and applying the Cameron-Martin-Girsanov theorem in stochastic calculus to evaluate this likelihood ratio. As a final step, several parameters are tuned appropriately to obtain the optimal constant (in the statement of the theorems).
	\end{rem}
	
	
	
	\begin{rem}[On the proofs of part (b) of Theorems~\ref{opconstant} and \ref{opconstantcx}] Observe that the random fluctuation of $f({\bm t_0}) - \hat{\ell}({\bm t_0})$ is stochastically dominated by the difference between $f(\bm t_0)$ and any particular term (corresponding to a bandwidth parameter) in the supremum in \eqref{lhat4}. The latter in turn can be decomposed into three terms---a random term having Gaussian fluctuations, a bias-like term arising from the kernel estimator and the multiscale penalization term. 
		The main crux of the proof involves a delicate choice of the bandwidth parameter which balances the bias-like term and the multiscale penalization term, thereby enabling us to lower bound the probability in~\eqref{monotoneupp} (and~\eqref{optcond22}).
	\end{rem}

	\section{Simulation studies}\label{simulation studies}
	In this section we construct confidence bands for different shape-restricted regression functions. For simulation purposes, instead of the continuous white noise model (\ref{eq:Mdl}) we consider its discrete analogue as detailed below.
	
	Let us start with the connection to nonparametric regression on gridded design. Let $\bm x_1,\ldots, \bm x_n \in \R^d$ be an enumeration of the  $m \times \cdots \times m $ uniform grid  $G_m^d:=\{1/m,2/m,\ldots,(m-1)/m,1\}^d$ where $m^d = n$. Let us look at the following nonparametric regression model: 
	\begin{equation}\label{discretemodel}
		Y_i = f({\bm x_i}) + \epsilon_i, \qquad \mbox{ for } i=1,\ldots,n, \end{equation} 
	where $f:[0,1]^d \to \R$ is the unknown regression function and $ \epsilon_i$'s are i.i.d.~standard normal random variables. For a kernel function  $\psi:\R^d \to \R$ and $\bm h,\bm t \in G^d_m$, such that $\bm t-\bm h,\bm t+ \bm h \in G^d_m$ we can define a kernel estimator $\hat{f}_{\bm h}$ of $f$ as 
	\begin{equation*} 
		\hat{f}_{\bm h}(\bm t) := \frac{ \sum_{i: \bm x_i \in B_{\infty}(\bm t,\bm h) } Y_i \,\psi\big( (\bm x_i-\bm t)/ \bm h\big)}{\sum_{i: \bm x_i \in B_{\infty}(\bm t,\bm h) }  \,\psi\big( (\bm x_i-\bm t)/ \bm h\big)}
	\end{equation*} 
	where by $(u_1,\ldots, u_d)/(h_1,\ldots, h_d)$ we mean the vector $(u_1/h_1,\ldots, u_d/h_d)$. We can also define the standardized kernel estimator as  \begin{equation*} 
		\hat{\Psi}_n(\bm t,\bm h) := \frac{ \sum_{i: \bm x_i \in B_{\infty}(\bm t,\bm h) } Y _i \,  \psi\big( (\bm x_i-\bm t)/\bm h\big)}{\sqrt{\sum_{ i: x_i \in B_{\infty}(\bm t,\bm h) }\psi^2\big( (\bm x_i-\bm t)/\bm h\big)}}. 
	\end{equation*} 
	Then the multiscale statistic for this regression problem reduces to 
	\begin{equation}\label{eq:T_n}
		T_n(Y,\psi):= \sup_{\bm h \in G^d_m : \bm t-\bm h,\bm t+\bm h \in G^d_m} \sup_{\bm t \in G^d_m} |\hat{\Psi}_n(\bm t,\bm h)| -  \Gamma\left(|B_{\infty}(\bm t,\bm h) \cap G^d_m|\right)
	\end{equation}
	where $|B_{\infty}(\bm t,\bm h) \cap G^d_m|$ denotes the number of elements in $B_{\infty}(\bm t,\bm h) \cap G^d_m$ and $Y\equiv (Y_1,\ldots,Y_n)$.
	\begin{figure}
		\centering
		\includegraphics[scale=.30]{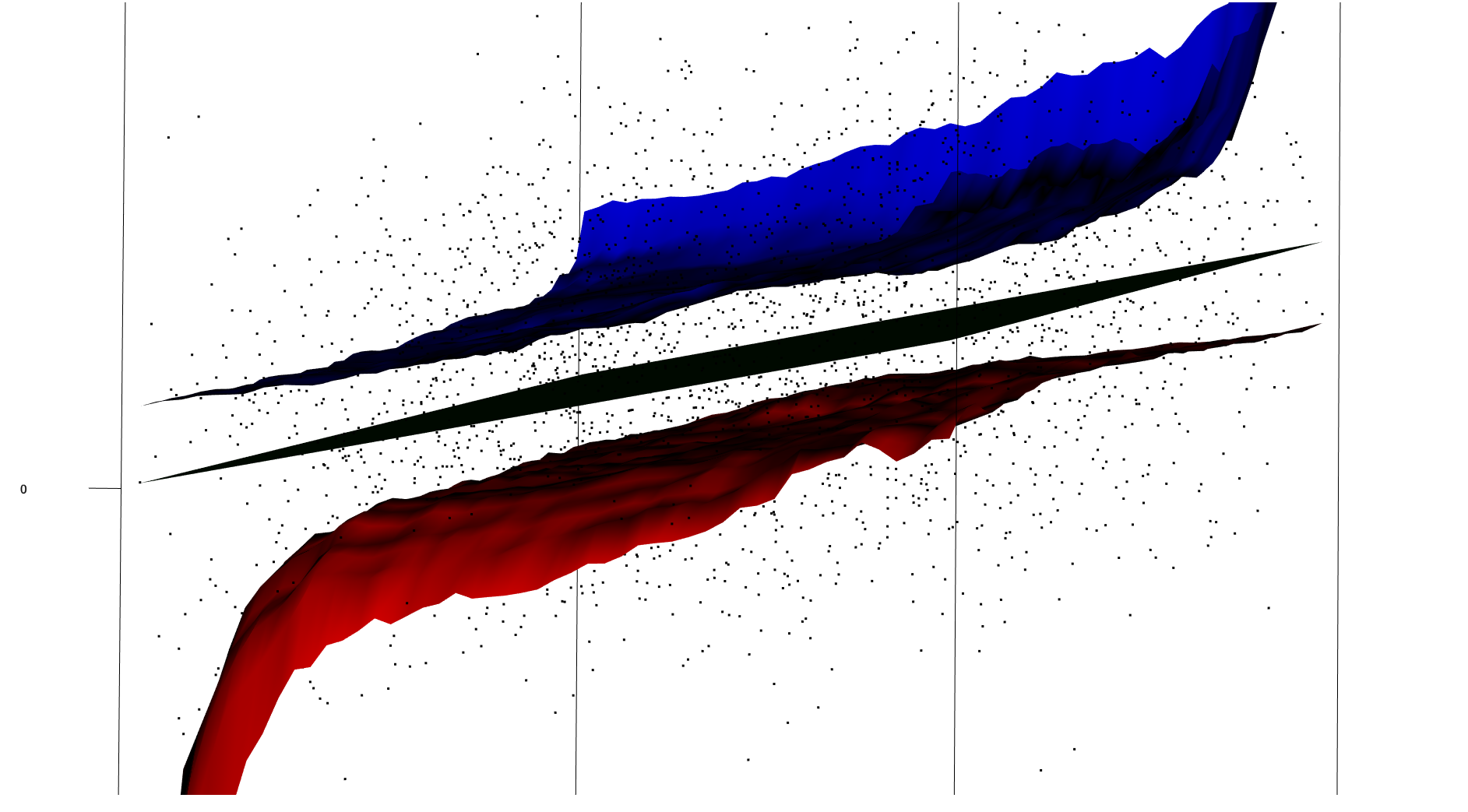}
		\caption{Confidence band for the function $f(x_1, x_2) = x_1 + x_2$ assuming $f$ is coordinate-wise isotonic
			and $n = 50^2$.}
		\label{Isotonic Confidence Bands-1}
	\end{figure}
	
	For our simulation studies we  consider $d=2$. At first we consider the regression function $f(x_1,x_2)=x_1+x_2$. In our simulation studies we consider data on a $50 \times 50$ grid on $[0,1]^2$ and assume that the underlying regression function is coordinate-wise nondecreasing. Figure~\ref{Isotonic Confidence Bands-1} shows our constructed confidence band with nominal coverage probability 0.95. Here we would like to point out that the confidence band has smaller width around the center of the rectangle and the width gets larger as we move towards the sides; this is expected as there are smaller number of data points close by to average over near the corners.
	\begin{figure}
		\centering
		\includegraphics[scale=.30]{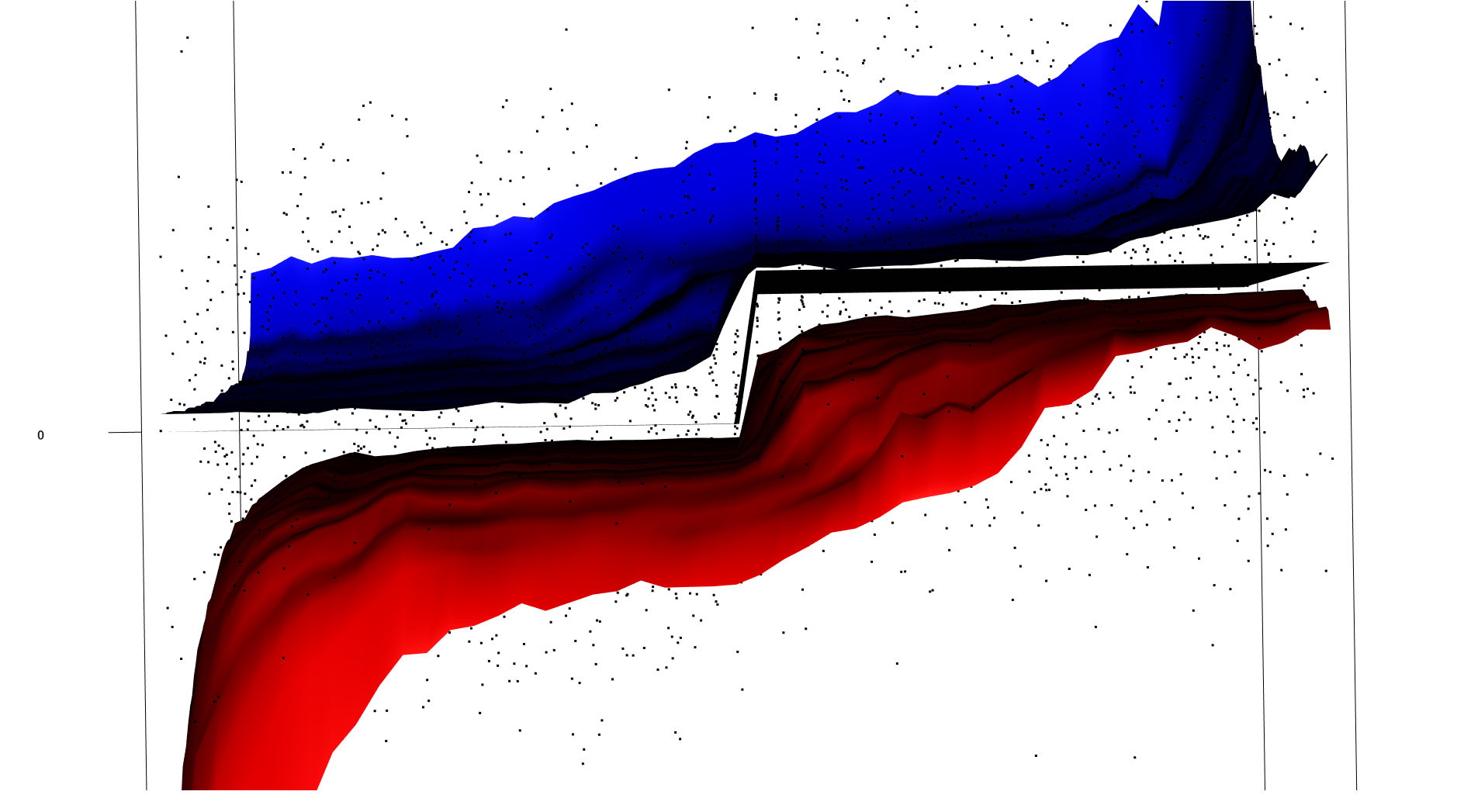}
		\caption{Confidence band for the function $f(x_1, x_2) = \mathbb{I}(x_1\geq0.5)$ assuming $f$ is coordinate-wise isotonic
			and $n = 50^2$.}
		\label{Isotonic Confidence Bands-2}
	\end{figure}
	
	In our second simulation study, we construct a confidence band for the function $f(x_1, x_2) = \mathbb{I}(x_1 \geq 0.5)$; see Figure~\ref{Isotonic Confidence Bands-2}. We can clearly see the local adaptivity of our band in action here. On the regions where the function is constant (regions where $x_1$ is away from 0.5) we see that the confidence band has significantly smaller width.

	\begin{figure}
		\centering
		\includegraphics[scale=.35]{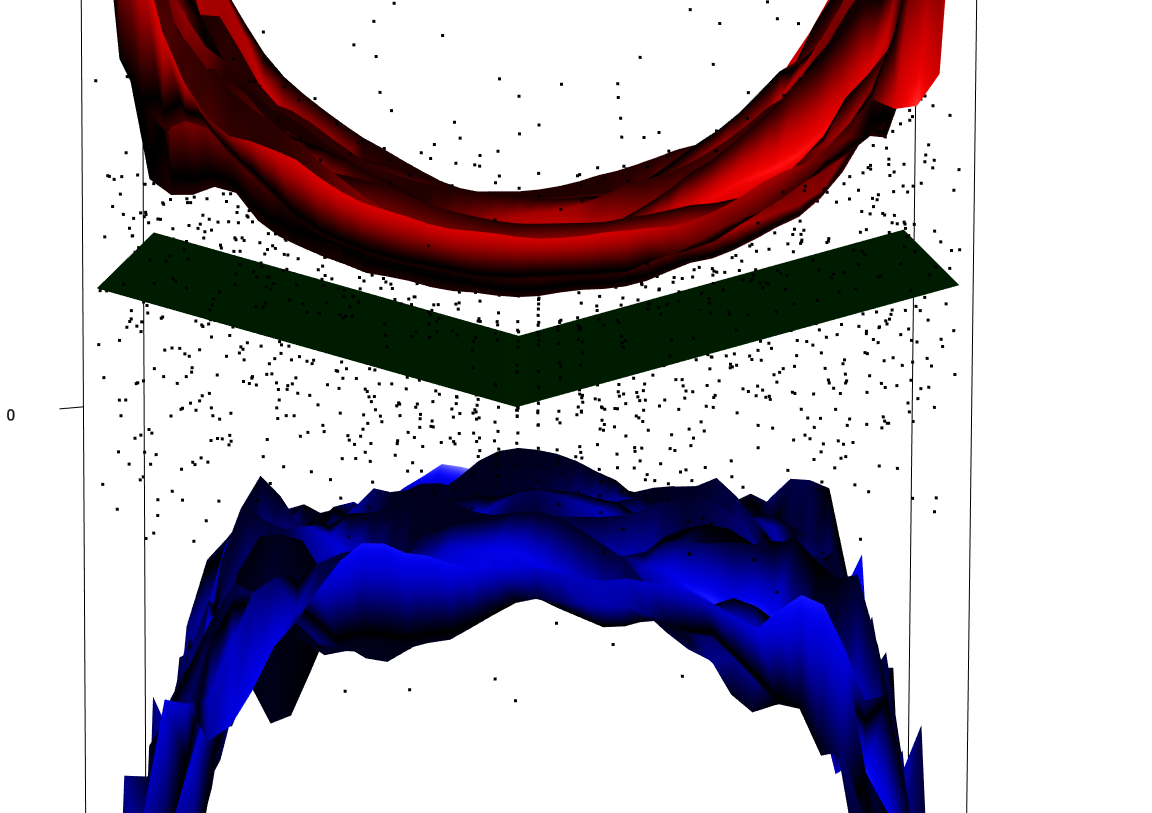}
		\caption{Confidence band for the function $f(x_1, x_2) = |x_1-0.5|$ assuming $f$ is convex
			and $n = 40^2$.}
		\label{Convex Confidence Bands-1}
	\end{figure}
	
	We also construct confidence bands under the assumption that the regression function is convex. Figure~\ref{Convex Confidence Bands-1} shows the  constructed confidence band for the regression function $f(x_1, x_2) = |x_1-0.5|$, whereas Figure~\ref{Convex Confidence Bands-2} corresponds to the regression function $f(x_1, x_2) = 40\{(x_1-0.5)^2+(x_2-0.5)^2\}$. Both of these plots show that the upper confidence band is much closer to the true function than the lower band; this behaviour is expected as we have shown in Theorem~\ref{opconstantcx} that the optimal separation constant for the upper bound $\Delta^{(u)} \approx 0.68278$ is smaller than that for the lower bound $\Delta^{(\ell)} \approx 1.19788$. 
	We also see that the lower confidence band is closer to the actual function around the center of the rectangle than on the sides. 
	\begin{figure}
		\centering
		\includegraphics[scale=.35]{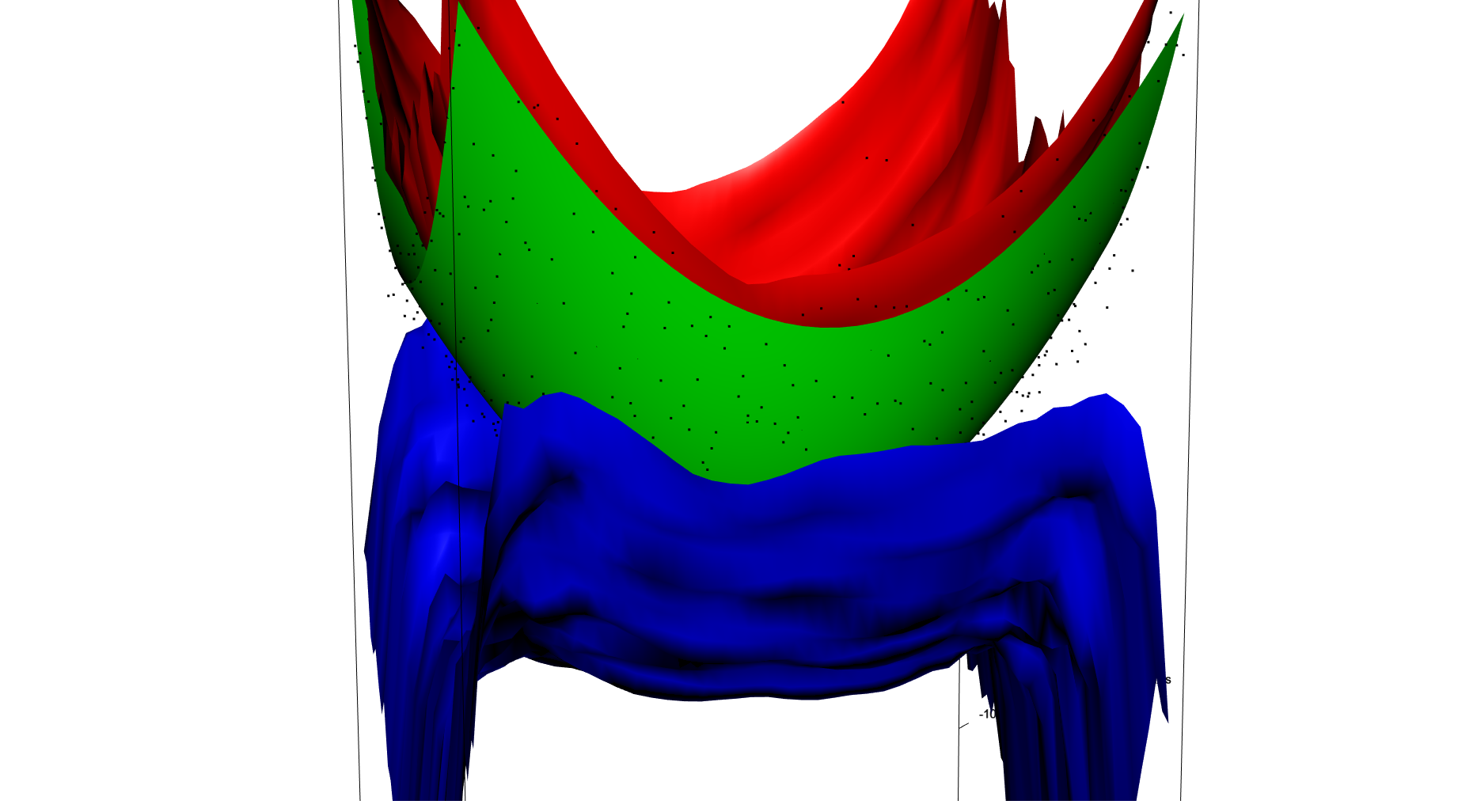}
		\caption{Confidence band for the function $f(x_1, x_2) = 40\{(x_1-0.5)^2+(x_2-0.5)^2\}$ assuming $f$ is convex
			and $n = 40^2$.}
		\label{Convex Confidence Bands-2}
	\end{figure}
	
	Table~\ref{Confidence Coverage} gives the estimated coverage probabilities of our constructed confidence bands for various coordinate-wise nondecreasing and convex functions. It shows that in at least $95\%$ of the cases our confidence bands do contain the actual regression function (which is expected as we have guaranteed converge even in finite samples; cf.~Theorem~\ref{th1}). Moreover, when the underlying function is a constant (if $f \in \mathcal{F}_1$) or affine (if $f \in \mathcal{F}_2$) we do observe that the nominal coverage coincides with the expected coverage (cf.~Proposition~\ref{exactcov}).
	
	\begin{table}
		\centering
		\begin{tabular}{| c | c | c| c| c|}
			\hline 
			$f(x_1,x_2)$& Class & Coverage probability  \\
			\hline
			0 & isotonic & 0.95 \\
			$x_1+x_2$& isotonic & 0.97 \\
			$20(x_1+x_2)$& isotonic & 1.00 \\
			$\mathbb{I}(x_1\geq 0.5)$& isotonic & 0.97\\
			\hline
			0 & convex & 0.95 \\
			$x_1+x_2$& convex & 0.96 \\
			$10(x_1+x_2)$& convex & 0.95 \\
			$(x_1-0.5)^2+(x_2-0.5)^2$& convex & 0.98\\
			\hline
		\end{tabular} 
		\caption{Estimated coverage probabilities of our constructed confidence bands for different regression functions.}
		\label{Confidence Coverage}
		
	\end{table}

	\section{Proofs of the main results}\label{prmainthmbd}
	In this section, we prove the main results of this paper. 
	
	\subsection{Proof of Proposition \ref{exactcov}}\label{pf:exactcov}
	We begin by showing that under the hypothesis of Proposition \ref{exactcov}, we have:
	\begin{equation}\label{symmtr}
		\e (\hat{f}_{\bm h}^\ell(\bt) )= f(\bt) = \e (\hat{f}_{\bm h}^u(\bt))\quad\text{for all}~\bt \in A_{\bh}, \text{ for all}~\bh \in I := (0,1/2]^d.
	\end{equation}
	Towards this, note that since $f,-f\in \mathcal{F}$, and since $-\hat{f}_{\bm h}$ is the kernel estimator of $-f$ as defined in \eqref{kerequ-0} (since for a standard $d$-dimensional Brownian sheet $W$, we have $W \stackrel{D}{=} -W$), we have:
	$$\E(\hat{f^{\ell}_{\bm h}}(\bt)) \leq f(\bm t) \leq \E(\hat{f^u_{\bm h}}(\bm t))\qquad\text{and}\qquad \E(-\hat{f^{\ell}_{\bm h}}(\bt)) \leq -f(\bm t) \leq \E(-\hat{f^u_{\bm h}}(\bm t))$$
	for all $\bm h \in I, ~\bm t\in A_{\bm h}$. This proves \eqref{symmtr}.
	
	Next, observe that in view of \eqref{guarantee1}, all it requires to complete the proof of Proposition \ref{exactcov} is to show that:
	\begin{equation}\label{gr22}
		\left \{\hat{\ell}(\bm t) \leq f(\bm t) \leq \hat{u}(\bm t)~\text{for all}~\bm t\in [0,1]^d \right\} \subseteq \left\{T(\psi^\ell) \leq \kappa_\alpha~,~T(-\psi^u)\leq \kappa_\alpha \right\}.
	\end{equation}
	Towards this, suppose that $\hat{\ell}(\bm t) \leq f(\bm t) \leq \hat{u}(\bm t)~\text{for all}~\bm t\in [0,1]^d$. In view of \eqref{symmtr}, we have:
	\begin{eqnarray*}
		&& \hat{\ell}(\bm t) \leq f(\bm t) \qquad \qquad\,\text{for all}~\bm t\in [0,1]^d\\&\implies& \hat{\ell}(\bm t) \leq \e (\hat{f}_{\bm h}^\ell(\bt)) \qquad~\text{for all}~\bm t\in [0,1]^d\\&\implies& \hat{f}^{\ell}_h(\bt) - \frac{\|\psi^\ell\| \left(\kappa_\alpha + \Gamma(2^d \prod_{i=1}^d h_i)\right)}{\langle 1,\psi^\ell \rangle (n\prod_{i=1}^dh_i)^{1/2}}\leq \e \hat{f}_{\bm h}^\ell(\bt)\qquad (\text{for all}~\bm h \in I, ~\bt \in A_{\bh})\\&\implies& \frac{\hat{f}^{\ell}_h(\bt) - \e \hat{f}_{\bm h}^\ell(\bt)}{\sqrt{\mathrm{Var}(\hat{f^{\ell}_h}(\bt))}} - \Gamma(2^d \prod_{i=1}^d h_i) \le \kappa_\alpha \qquad \qquad (\text{for all}~\bm h \in I, ~\bt \in A_{\bh})\\&\implies& T(\psi^\ell) \le \kappa_\alpha.
	\end{eqnarray*}
	Similarly, one has:
	\begin{eqnarray*}
		&& \hat{u}(\bm t) \ge f(\bm t) \qquad\,\qquad\text{for all } \bm t\in [0,1]^d\\&\implies& \hat{u}(\bm t) \ge \e (\hat{f}_{\bm h}^u(\bt)) \qquad \, \text{for all } \bm t\in [0,1]^d\\&\implies& \hat{f^{u}_h}(\bt) + \frac{\|\psi^u\| \left(\kappa_\alpha + \Gamma(2^d \prod_{i=1}^d h_i)\right)}{\langle 1,\psi^u \rangle (n\prod_{i=1}^dh_i)^{1/2}}\ge \e \hat{f}_{\bm h}^u(\bt)~\qquad  (\text{for all}~\bm h \in I, ~\bt \in A_{\bh})\\&\implies& \frac{\hat{f^{u}_h}(\bt) - \e \hat{f}_{\bm h}^u(\bt)}{\sqrt{\mathrm{Var}(\hat{f^{u}_h}(\bt))}} + \Gamma(2^d \prod_{i=1}^d h_i) \ge -\kappa_\alpha\qquad  \qquad (\text{for all}~\bm h \in I, ~\bt \in A_{\bh})\\&\implies& T(-\psi^u) \le \kappa_\alpha.
	\end{eqnarray*}
	This proves \eqref{gr22} and completes the proof of Proposition \ref{exactcov}. \qed

	\subsection{Proof of Theorem \ref{spadth}}\label{prspadth}
	To prove the first result, we take $\bm h := \varepsilon \boldsymbol{1}_d$. Note that the hypothesis \eqref{eqty} of Proposition \ref{spadth} implies that for all $\bm t \in D$,
	\begin{equation}\label{eq:Bias-0}
		\frac{\langle f(\bt+\bh\star\cdot)-f(\bt),\psi^u(\cdot)\rangle}{\langle1,\psi^u\rangle} = \E (\hat{f}_{\bm h}^u(\bt)) - f(\bt) = 0.
	\end{equation}
	Also note that, for any ${\bm t} \in A_{\varepsilon \mathbf{1}_d}$,
	\begin{eqnarray}
		&& \hat{u}(\bt)-f(\bt) \nonumber\\
		&\leq& \hat{f_{\bm h}^u}(\bt)- f(\bt) + \frac{\|\psi^u\|\left(\kappa_\alpha+\Gamma(2^d\prod_{i=1}^dh_i)\right)}{\langle 1 , \psi^u \rangle (n\prod_{i=1}^dh_i)^{1/2} }\nonumber\\
		&\leq& \frac{\langle f(\bt+\bh\star\cdot)-f(\bt),\psi^u(\cdot)\rangle}{\langle1,\psi^u\rangle} + \frac{\|\psi^u\|(\kappa_\alpha+2\Gamma(2^d(\prod_{i=1}^dh_i))+T(\psi^u))}{\langle 1 , \psi^u \rangle (n\prod_{i=1}^dh_i)^{1/2}}. \label{eq:Upper-bound-U}
	\end{eqnarray}
	For any $\bt \in D \subset A_{\varepsilon \mathbf{1}_d}$, we thus have, using~\eqref{eq:Bias-0},
	\begin{equation*}\label{ufbound}
		\hat{u}(\bt) - f(\bt) \le n^{-1/2}\frac{\|\psi^u\|(\kappa_\alpha+2\Gamma(2^d\varepsilon^d)+T(\psi^u))}{\langle 1 , \psi^u \rangle \varepsilon^{d/2}} \le K_{\varepsilon} n^{-1/2}\left(|\kappa_\alpha|/2 + |T(\psi^u)|\right)
	\end{equation*}
	for some constant $K_{\varepsilon} > 0$. A similar analysis can be done for $f(\bt) - \hat{\ell}(\bt)$, which concludes the proof of the first result by adding the two bounds.

	To prove the second result, we take $\varepsilon \equiv \varepsilon_n  := (\log (en))^{-\frac{1}{d}}$, we can conclude from~\eqref{eq:Upper-bound-U} (and using~\eqref{eq:Bias-0}) that for $\bt \in D_n$, 
	\begin{eqnarray*}\label{ufboundd7}
		\hat{u}(\bt) - f(\bt) &\le& K_1 (\log(en))^{1/2} n^{-1/2} \|\psi^u\|\left(|\kappa_\alpha| + |T(\psi^u)| + \sqrt{\log \log (en)}\right)/\langle 1,\psi^u\rangle\nonumber\\&=& K_{2}\left(\frac{\log(en) \log \log (en)}{n}\right)^{1/2} \left(1+\frac{|\kappa_\alpha| + |T(\psi^u)|}{\sqrt{\log \log (en)}}\right)
	\end{eqnarray*}
	for some constants $K_1, K_{2}>0$. The bound for $f(\bt) - \hat{\ell}(\bt)$ follows similarly, thereby completing the proof of Theorem~\ref{spadth}. \qed

	\subsection{Proof of Theorem \ref{thm2}}\label{prmainthm}
	For $j=1$, we will show that just the facts that $\psi_1^u$ is supported on a subset of $[0,\infty)^d$, $\psi_1^\ell$ is supported on a subset of $(-\infty,0]^d$, and they are non negative, are enough to conclude Theorem \ref{thm2}. This follows easily from the fact that the coordinate-wise nondecreasing nature of $f$ ensures that:
	$$\langle f(\bt+\bh\star\cdot),\psi_1^u\rangle \ge f(\bt) \langle 1,\psi_1^u\rangle\qquad\text{and}\qquad \langle f(\bt+\bh\star\cdot),\psi_1^\ell\rangle \leq f(\bt) \langle 1,\psi_1^\ell\rangle.$$
	Next, we consider the case $k=2$. If we could show that for every convex function $g:\R^d\to \R$, we have:
	\begin{equation}\label{short}
		\langle g, \psi_2^u\rangle \ge g(\boldsymbol 0) \langle 1,\psi_2^u\rangle\qquad\qquad\text{and}\qquad \qquad\langle g, \psi_2^\ell\rangle \leq g(\boldsymbol 0) \langle 1,\psi_2^\ell\rangle,
	\end{equation}
	then we would be done, because then substituting $g(\bm x) := f(\bt + \bh \star \bx)$ (which is a convex function) in \eqref{short} will complete the proof. We can also assume, without loss of generality, that $g(\bm 0)=0$, because otherwise we can apply \eqref{short} on the function $g(\cdot)-g(\bm 0)$.  In view of all these reductions, we just need to show that $\langle g, \psi_2^u\rangle \ge 0$ and $\langle g,\psi_2^\ell\rangle \leq 0$. The first inequality in~\eqref{short} is a direct consequence of Jensen's inequality, because if $\bm U$ denotes a random vector distributed on the $d$-dimensional sphere $B_{d} := \{\bm x \in \R^d: \|\bx\|\leq 1\}$, with density at $\bm x$ being proportional to $1-\|\bx\|^2$, then there exists a constant $C>0$ such that:
	$$\langle g, \psi_2^u\rangle = C \e [g(\bm U)] \ge Cg(\e [\bm U]) = C g(\bm 0) = 0$$
	where we used the fact that $\e [\bm U] = 0$ by symmetry of the distribution of $\bm U$ around $\bm 0$.
	
	Finally, to prove that $\langle g,\psi_2^\ell\rangle \leq 0$, first note that by convexity of $g$ and the fact that $g(\bm 0) = 0$, we have:
	$$g(\alpha \bm y) \leq \alpha g(\bm y)\quad\quad\text{for all}\quad \bm y\in \R^d~\text{and}~\alpha \in [0,1].$$
	We can now substitute $\alpha := (d+3)\|\bm x\|/(d+1)$ and $\bm y := (d+1)\bm x/((d+3)\|\bm x\|)$, and have:
	$$g\left(\frac{(d+1)\bm x}{(d+3)\|\bm x\|}\right)\ge \frac{(d+1)}{(d+3)\|\bm x\|} g(\bm x)\quad\quad\text{when}\quad\|\bm x\|\leq \frac{d+1}{d+3}~.$$
	Similarly, we can substitute $\alpha := (d+1)/((d+3)\|\bm x\|)$ and $\bm y := \bm x$, and have:
	$$g\left(\frac{(d+1)\bm x}{(d+3)\|\bm x\|}\right)\leq \frac{(d+1)}{(d+3)\|\bm x\|} g(\bm x)\quad\quad\text{when}\quad\frac{d+1}{d+3} \leq \|\bm x\|\leq 1~.$$
	Moreover, note that $\psi_2^\ell(\bm x) \leq 0$ when $(d+1)/(d+3) \leq \|\bm x\| \leq 1$ and $\psi_2^\ell(\bm x) \ge 0$ when $\|\bm x\| \leq (d+1)/(d+3)$. We have:
	\begin{eqnarray*}
		\langle g,\psi_2^\ell \rangle &=& \int_{B_{d}} \left(1-\frac{2d+4}{d+1} \|\bx\| + \frac{d+3}{d+1}\|\bx\|^2\right)g(\bm x) \,d\bm x\\
		&\leq& \frac{d+3}{d+1}\int_{B_{d}} \|\bm x\|\left(1-\frac{2d+4}{d+1} \|\bx\| + \frac{d+3}{d+1}\|\bx\|^2\right) g\left(\frac{(d+1)\bm x}{(d+3)\|\bm x\|}\right) \,d\bm x~.
	\end{eqnarray*}
	At this point, for every $\bm e \in \{-1,1\}^d$, define:
	$$H_{\bm e}:= \{\bm x \in B_{d}: e_i x_i\ge 0~\text{for all}~1\leq i\leq d\}~.$$
	Note that $\{H_{\bm e}\}_{\bm e \in \{-1,1\}^d}$ form the $2^d$ orthants of $\mathbb{R}^d$, intersected with $S_{d-1}$. We will show that for all $\bm e \in \{-1,1\}^d$,
	\begin{equation}\label{orthbreak}
		\int_{H_{\bm e}} \|\bm x\|\left(1-\frac{2d+4}{d+1} \|\bx\| + \frac{d+3}{d+1}\|\bx\|^2\right) g\left(\frac{(d+1)\bm x}{(d+3)\|\bm x\|}\right) d\bm x=0 
	\end{equation}
	which is enough to complete the proof. Towards this, fix $\bm e \in \{-1,1\}^d$, and
	make the following change of variables $\bm x \mapsto \bm y := (y_0,y_1,\ldots,y_{d-1})$ on $H_{\bm e}$:
	$$y_0= \|\bm x\|\qquad\text{and}\qquad y_i = \frac{x_i}{\|\bm x\|}~\text{for all}~ 1 \leq i \leq d-1~.$$
	This transformation is invertible, and we have:
	$$x_i = y_0y_i~\quad \text{for all}~\quad 1 \leq i \leq d-1\qquad\text{and}\qquad x_d := y_0 e_d \sqrt{1-y_1^2-\ldots - y_{d-1}^2}~.$$
	The Jacobian of this transformation is given by:
	
	\[ J(\bm y) = 
	\begin{bmatrix}
		y_1 & y_0 & 0 &\ldots & 0\\
		y_2 & 0 & y_0 &\ldots & 0\\
		\vdots & \vdots & \vdots &\vdots &\vdots\\
		e_d\sqrt{1-\sum_{i=1}^{d-1} y_i^2} & -\frac{y_0y_1e_d}{\sqrt{1-\sum_{i=1}^{d-1} y_i^2}} & -\frac{y_0y_2s_d}{\sqrt{1-\sum_{i=1}^{d-1} y_i^2}} & \ldots & -\frac{y_0y_{d-1}e_d}{\sqrt{1-\sum_{i=1}^{d-1} y_i^2}}
	\end{bmatrix}
	\]
	and hence, we have:
	$$|\mathrm{det}(J(\bm y))| = \frac{y_0^{d-1}}{\sqrt{1-\sum_{i=1}^{d-1}y_i^2}}.$$
	Therefore, defining $\tilde{\bm y} := \left(y_1,\ldots, y_{d-1}, e_d\sqrt{1-y_1^2-\ldots -y_{d-1}^2}\right)$, we have:
	\begin{eqnarray*}
		&&\int_{H_{\bm e}} \|\bm x\|\left(1-\frac{2d+4}{d+1} \|\bx\| + \frac{d+3}{d+1}\|\bx\|^2\right) g\left(\frac{(d+1)\bm x}{(d+3)\|\bm x\|}\right) d\bm x\\&=& \int_{B_{d-1}\bigcap\prod_{i=1}^{d-1} e_i[0,1]} \frac{g\left(\frac{d+1}{d+3} ~\tilde{\bm y}\right)}{\sqrt{1-\sum_{i=1}^{d-1}y_i^2}}\int_0^1 y_0^d\left(1-\frac{2d+4}{d+1} y_0 + \frac{d+3}{d+1}y_0^2\right)~dy_0\, dy_1\,\ldots \,d y_{d-1}\\&=& 0\quad\quad (\text{since the inner integral is} ~0).
	\end{eqnarray*}
	This proves \eqref{orthbreak} and completes the proof of Theorem \ref{thm2}. \qed

	\subsection{Proof of Theorem \ref{Thm 2}}\label{thm2prf}
	To begin with, note that for $\bt \in A_{\bm h}$, we have 
	\begin{eqnarray}\label{eq6}
		&& \hat{u}(\bt)-f(\bt) \nonumber\\
		&\leq& \hat{f_{\bm h}^u}(\bt)- f(\bt) + \frac{\|\psi^u\|\left(\kappa_\alpha+\Gamma(2^d\prod_{i=1}^dh_i)\right)}{\langle 1 , \psi^u \rangle (n\prod_{i=1}^dh_i)^{1/2} }\nonumber\\
		&\leq& \frac{\langle f(\bt+\bh\star\cdot)-f(\bt),\psi^u(\cdot)\rangle}{\langle1,\psi^u\rangle} + \frac{\|\psi^u\|(\kappa_\alpha+2\Gamma(2^d(\prod_{i=1}^dh_i))+T(\psi^u))}{\langle 1 , \psi^u \rangle (n\prod_{i=1}^dh_i)^{1/2}}.
	\end{eqnarray}
	Here the last line follows from the inequality $$T(\psi^u) \geq \frac{\hat{f_{\bh}^u}(\bt)-\langle f(\bt+\bh\star\cdot),\psi^u(\cdot)\rangle/\langle1,\psi^u\rangle}{\|\psi\|\langle1,\psi^u\rangle^{-1}(n\prod_{i=1}^dh_i)^{-1/2}} - \Gamma(2^d\prod_{i=1}^d h_i).$$
	Now, if $f \in \mathbb{H}_{\beta,L}\cap \mathcal{F}_1$ (where $0<\beta\leq 1$) we have 
	\begin{eqnarray*}
		|f(\bt+\bh\star \bx)-f(\bt)|  \leq L\|\bh\star \bx\|^\beta .
	\end{eqnarray*}
	Hence, denoting $h := \max\{h_1,\ldots,h_d\}$, we have:
	\begin{equation}\label{sh1}
		\frac{\langle f(\bt+\bh\star\cdot)-f(\bt),\psi^u(\cdot)\rangle}{\langle1,\psi^u\rangle} \leq \frac{L\int_{[-1,1]^d} \|\bh \star \bx\|^\beta \psi^u(\bx) d\bx}{\langle 1,\psi^u\rangle} \leq \frac{Lh^\beta \int_{[-1,1]^d} \|\bx\|^\beta \psi^u(\bx) dx}{\langle 1,\psi^u\rangle} := K_1 h^\beta
	\end{equation}
	where $K_1 := L\int_{[-1,1]^d} \|\bx\|^\beta \psi^u(\bx) dx/\langle 1,\psi^u\rangle$.
	On the other hand, if $f\in \mathbb{H}_{\beta,L}\cap \mathcal{F}_2$ (where $1<\beta\leq 2$) then defining $g(\bm x) := f(\bt + \bh \star \bx)$, we have the following for some $\xi_{\bx}$ lying in the segment joining $\boldsymbol{0}$ and $\bx$:
	\begin{eqnarray}\label{sh2}
		\langle f(\bt+\bh\star\cdot)-f(\bt),\psi^u(\cdot)\rangle&=&\int_{[-1,1]^d} \bx^\top \nabla g(\xi_{\bx}) \psi^u(\bx) \,d\bx\nonumber\\
		&=& \int_{[-1,1]^d} \bx^\top \left(\nabla g(\xi_{\bx})-\nabla g(\boldsymbol 0)\right) \psi^u(\bx) \,d\bx\nonumber\\
		&\leq & \int_{[-1,1]^d} \|\bx\| \|\nabla g(\xi_{\bx}) - \nabla g(\boldsymbol 0)\| \psi^u(\bx) \,d\bx\nonumber\\&=& \int_{[-1,1]^d} \| \bx\|  \|\bm h\star (\nabla f(\bt + \bh\star \xi_{\bx}) - \nabla f(\bt))\| \psi^u(\bx) \,d\bx\qquad\\
		&\leq & h \int_{[-1,1]^d} \| \bx\|  \|\nabla f(\bt + \bh\star \xi_{\bx}) - \nabla f(\bt)\| \psi^u(\bx) \,d\bx\nonumber\\
		&\leq& Lh^\beta \int_{[-1,1]^d} \|\bx\|^\beta \psi^u(\bx) \,d\bx~.\nonumber
	\end{eqnarray}
	Note that in the second equality we have used the fact that both the integrals $\int_{[-1,1]^d} x_i \psi^u(\bm x)$ and $\int_{[-1,1]^d} - x_i \psi^u(\bm x)$ are nonnegative for $1\le i\le d$, which follows from the bias condition \eqref{mainkercondition} (as the functions $f({\bm x}) = \pm x_i$ are convex, for $i=1,\ldots, d$) and hence $\int_{[-1,1]^d} x_i \psi^u(\bm x) = 0$. The last inequality follows from the fact that $f\in \mathbb{H}_{\beta,L}$, and hence, $$\|\nabla f(\bt + \bh\star \xi_{\bx}) - \nabla f(\bt)\| \le L \|\bm h \star \xi_{\bm x}\|^{\beta-1} \le L h^{\beta-1}.$$ Hence, in this case also, we have:
	$$\frac{\langle f(\bt+\bh\star\cdot)-f(\bt),\psi^u(\cdot)\rangle}{\langle1,\psi^u\rangle} \leq \frac{Lh^\beta \int_{[-1,1]^d} \|\bx\|^\beta \psi^u(\bx) \, dx}{\langle 1,\psi^u\rangle} = K_1 h^\beta.$$
	Therefore, \eqref{eq6} tells us that as long as ${\bm t} \in A_{\bm h}$ we have 
	\begin{equation}\label{indilt}
		\hat{u}(\bt)-f(\bt) \leq K_1 h^\beta + \frac{\|\psi^u\|(\kappa_\alpha+2\Gamma(2^d(\prod_{i=1}^dh_i))+T(\psi^u))}{\langle 1 , \psi^u \rangle (n\prod_{i=1}^dh_i)^{1/2} }. 
	\end{equation}
	Putting $h_1=h_2=\ldots=h_d=\varepsilon_n:=(\log(en)/n)^{1/(2\beta+d)}$ in \eqref{indilt}, we get
	$K_1 h^\beta = K_1 \varepsilon_n^\beta$ and
	\begin{eqnarray*}
		\frac{\|\psi^u\|(\kappa_\alpha+2\Gamma(2^d(\prod_{i=1}^dh_i))+T(\psi^u))}{\langle 1 , \psi^u \rangle (n\prod_{i=1}^dh_i)^{1/2} }
		&\leq& K_2 \frac{|\kappa_\alpha|+|T(\psi^u)| +2\sqrt{2+2\log n}}{n^{\beta/(2\beta+d)} \log(en)^{d/(4\beta+2d)}}\\
		& \leq & K_3\varepsilon_n^\beta \left(\frac{|\kappa_\alpha|+|T(\psi^u)|}{\log^{1/2}(en)}+1\right)
	\end{eqnarray*}
	for some constants $K_2$ and $K_3$ not depending on $n$.
	The above two equations tell us that as long as ${\bm t} \in A_{\varepsilon_n\boldsymbol{1}_d}$, we have 
	\begin{equation}\label{bd31}
		\hat{u}(\bt)-f(\bt) \leq K \varepsilon_n^\beta \left(1+ \frac{|\kappa_\alpha|+|T(\psi^u)|}{\log^{1/2}(en)}\right)   
	\end{equation}
	for some constant $K$ not depending on $n$.
	
	The steps for bounding $f(\bm t) - \hat{\ell}(\bm t)$ are similar, but we point out some differences. First, we have:
	$$f(\bm t) - \hat{\ell}(\bm t) \le \frac{\langle f(\bt) - f(\bt+\bh\star\cdot),\psi^\ell(\cdot)\rangle}{\langle1,\psi^\ell\rangle} + \frac{\|\psi^\ell\|(\kappa_\alpha+2\Gamma(2^d(\prod_{i=1}^dh_i))+T(-\psi^\ell))}{\langle 1 , \psi^\ell \rangle (n\prod_{i=1}^dh_i)^{1/2} },$$ which gives:
	$$f(\bm t) - \hat{\ell}(\bm t) \le K_2 h^\beta +  \frac{\|\psi^\ell\|(\kappa_\alpha+2\Gamma(2^d(\prod_{i=1}^dh_i))+T(-\psi^\ell))}{\langle 1 , \psi^\ell \rangle (n\prod_{i=1}^dh_i)^{1/2} }$$
	where $K_2 := L\int_{[-1,1]^d} \|\bx\|^\beta |\psi^\ell(\bx)| \,d{\bm x}/\langle 1,\psi^\ell\rangle$. Hence, for some constant $K'$ not depending on $n$, we have:
	\begin{equation}\label{bd32}
		f(\bm t) - \hat{\ell}(\bm t) \le K'  \varepsilon_n^\beta \left(1+ \frac{|\kappa_\alpha|+|T(-\psi^\ell)|}{\log^{1/2}(en)}\right).
	\end{equation}
	Theorem \ref{Thm 2} now follows by adding \eqref{bd31} and \eqref{bd32}. \qed
	
	\subsection{Sketch of a proof of Theorem \ref{locad}}\label{sec:Loc-Smooth}
	For proving Theorem \ref{locad}, first one needs to observe that for $\bm h = \varepsilon \boldsymbol{1}_d$, we have $\|\bm h \star \bm x\|_\infty  = \varepsilon \|\bm x\|_\infty \le \varepsilon$, and hence, $\bm t_0 + \bm h \star \bm x \in {B}_\infty(\bm t_0, \varepsilon)$. The rest of the proof follows exactly as the proof of Theorem \ref{Thm 2}, on noting that one only needs the H\"older smoothness assumption for bounding the terms $|f(\bt_0 + \bm h\star \bm x) - f(\bt_0)|$ and $\|\nabla f(\bt_0 + \bm h \star \xi_{\bm x}) - \nabla f(\bt_0)\|$ for some $\xi_{\bm x}$ lying in the segment joining $\boldsymbol{0}$ and $\bm x$ (for the classes $\mathcal{F}_1$ and $\mathcal{F}_2$ respectively). Consequently, it is enough to have H\"older smoothness on ${B}_\infty(\bm t_0, \varepsilon)$ only.

	\subsection{Proof of Theorem \ref{idimth}}\label{pridim}
	It follows from \eqref{eq6} and \eqref{indilt} that for $\bm t \in A_{\boldsymbol{\varepsilon}_n,i_1,\ldots,i_k}$, we have:
	\begin{eqnarray*}\label{indilt2}
		\hat{u}(\bt)-f(\bt) &\leq& K_1 \varepsilon_n^\beta + \frac{\|\psi^u\|(\kappa_\alpha+2\Gamma(2^d\varepsilon_n^k \varepsilon^{d-k})+T(\psi^u))}{\langle 1 , \psi^u \rangle (n\varepsilon_n^k \varepsilon^{d-k})^{1/2} }\nonumber\\&=& K_1 \varepsilon^\beta\rho_{n,k} + \frac{\|\psi^u\|(\kappa_\alpha+2\Gamma(2^d\varepsilon^d (\log(en)/n)^{k/(2\beta+k)})+T(\psi^u))}{\langle 1 , \psi^u \rangle (n\varepsilon^d (\log(en)/n)^{k/(2\beta+k)})^{1/2}}\nonumber\\&\le& \Delta_1 \rho_{n,k} \left(1+\frac{|\kappa_{\alpha}|+|T(\psi^u)|}{(\log(en))^{1/2}}\right)
	\end{eqnarray*}
	for some constants $K_1$ and $\Delta_1>0$. The rest of the proof follows the idea of the proof of Theorem \ref{Thm 2}. The only modifications are in \eqref{sh1} and \eqref{sh2}, where one now uses the fact that the function $f$ only depends on $k$ coordinates, and hence, the vector $\bh$ can now be replaced by its restriction on the $i_1^{\mathrm{th}},\ldots,i_k^{\mathrm{th}}$ coordinates.


	\subsection{Proof of Theorem \ref{opconstant}}\label{optl}
	
	\noindent (a)~We prove only the bound for $\|f- {\ell}\|_{U}$ as the other case can be handled similarly. Thus, we will show that for any level $1-\alpha$ confidence band $(\ell,u)$ with guaranteed coverage probability for the class $\mathcal{F}_1$, and any $0<\gamma<1$, we have  $$\liminf_{n \to \infty}\p_{f}\left(\|f- \ell\|_{U} \geq \gamma \Delta^{(\ell)} L_1^{\frac{d}{2+d}}[\bm t_0] \rho_n\right) \geq 1-\alpha.$$ For notational simplicity, we will abbreviate $\psi^\ell$ by $\psi$. By our assumption, $f$ is continuously differentiable on an open neighborhood $U$ of $\bm t_0 \in (0,1)^d$ such that $$L_1[\bm t_0] \equiv L_1[f,\bm t_0]:= \left[\prod_{i=1}^d \frac{\partial}{\partial x_i} f(\bm x) \big|_{\bm x=\bm t_0}\right]^{1/d} > 0.$$ Let us define, for $i=1,\ldots, d$,  $$M_i^\star:=\frac{\partial}{\partial x_i}f(\bm x)|_{\bm x=\bm t_0}.$$  Without loss of generality, let us assume that $M_1^\star \leq M_2^\star \leq \ldots \leq M_d^\star.$  Since $\gamma<1$, we can find $\epsilon>0$ and $\gamma^*<1$ (depending on $\gamma$) such that:
	\begin{equation}\label{eq:L_1-t0}
		\gamma  L_1^{d/(2+d)}[\bm t_0]=\gamma^\star( L_1[\bm t_0]-\epsilon)^{d/(2+d)}.
	\end{equation} 
	Also since $f$ is continuously differentiable on $U$, we can find $\bm h_0\in [0,1]^d$ and  $\epsilon^\star >0$ small enough such that:  \begin{itemize} 
		\item[(i)] $B_\infty(\bm t_0,\bm h_0) \subset U$, 
		
		\item[(ii)] $\left[\prod_{i=1}^d (M_i^\star-\epsilon^\star)\right]^{1/d} \geq {L}_1[\bm t_0]-\epsilon$,
		
		\item[(iii)] for all $\bm x \in B_\infty(\bm t_0,\bm h_0)$ we have $$\frac{\partial}{\partial x_i}f(\bm x) \geq M_i^\star- \epsilon^\star>0, \quad \quad \mbox{for all } \;\;i=1,2,\ldots,d.$$  
	\end{itemize}
	Suppose that $\bm h =(h_1,h_2,\ldots,h_d)$ is such that $h_1 \in (0,1/2]$ and $h_i := h_1\times \left(\frac{M_1^\star-\epsilon^\star}{M_i^\star-\epsilon^\star}\right)$, for $i=2,\ldots, d$. Let us define a set of grid points $G$ for bandwidth $\bm h$ as  $$G :=\left\{\bm t=(t_1,\ldots,t_d): t_i=t_{0i}+h_i(2k_i) \mbox{ for some integer } k_i  , B_\infty(\bm t,\bm h) \subset B_{\infty}(\bm t_0,\bm h_0) \right\}$$
	where $\bm t_0=(t_{01},\ldots,t_{0d})$. For $\bm t\in G$, let us define the following ``perturbation'' functions: $$f_{\bm t}:=f-h_1(M_1^\star-\epsilon^\star)\psi_{\bm t,\bm h},$$ where $\psi_{\bm t,\bm h}$ is defined in~\eqref{kerest}. We will now show that for every $\bm t \in G$, $f_{\bm t} \in \mathcal{F}_1$.
	
	\begin{lem}\label{ftinF1}
		$f_{\bm t} \in \mathcal{F}_1$ \textrm{ for all } $\bm t \in G.$
	\end{lem}
	\begin{proof}
		Fix $\bm t = (t_1,\ldots, t_d) \in G$. Suppose that $\bm x \le \bm y$ (coordinate-wise) and $\bm x,\bm y\in T_{\bm t,\bm h} $ where 
		$$T_{\bm t,\bm h} := \left\{\bm u=(u_1,\ldots,u_d) \in [0,1]^d: u_i\leq t_i \mbox{ for all } i \mbox{ and } \sum_{i=1}^d \left(\frac{u_i-t_i}{h_i}\right) \geq -1\right\}. $$Then, for some $\boldsymbol{\xi} \in [\bm x, \bm y]\subset B_{\infty}(\bm t_0,\bm h_0)$ (here $[\bm x, \bm y]$ denotes the line segment joining $\bm x$ and $\bm y$), by the mean value theorem (and (iii) above),
		\begin{equation}\label{ds11}
			f(\bm y)-f(\bm x) = \triangledown f(\boldsymbol{\xi})^\top (\bm y-\bm x) \; \geq \;  \sum_{i=1}^d (M_i^\star-\epsilon^\star)(y_i-x_i). 
		\end{equation}
		Then, using the formula for $\psi_{\bm t,\bm h}$,
		\begin{equation}\label{ds22}
			-\psi_{\bm t,\bm h}(\bm y)+\psi_{\bm t,\bm h}(\bm x) = -\sum_{i=1}^d \Big(\frac{y_i-x_i}{h_i}\Big) = -\frac{1}{h_1(M_1^\star-\epsilon^*)} \sum_{i=1}^d (M_i^\star-\epsilon^\star)(y_i-x_i).
		\end{equation}
		It now follows from \eqref{ds11} and \eqref{ds22} that:
		$$f_{\bm t}(\bm y) - f_{\bm t}(\bm x) \ge \sum_{i=1}^d (M_i^\star-\epsilon^\star)(y_i-x_i) -h_1(M_1^\star - \epsilon^\star) \left[\frac{1}{h_1(M_1^\star-\epsilon^*)} \sum_{i=1}^d (M_i^\star-\epsilon^\star)(y_i-x_i)\right] = 0$$
		thereby yielding $f_{\bm t}(\bm y) \geq f_{\bm t}(\bm x)$. 

		Let us now look at the case $\bm x\notin T_{\bm t,\bm h}$, $\bm y \in T_{\bm t,\bm h}$ and $\bm x\leq \bm y.$ Define $$a := 1+\sum_{i=1}^d \left(\frac{x_i-t_i}{h_i}\right) \qquad \mbox{and } \qquad b :=1+ \sum_{i=1}^d \left(\frac{y_i-t_i}{h_i}\right).$$ Since $\bm y \in T_{{\bm t},{\bm h}}$, we have $\bm x \leq \bm y \leq \bm t$ (coordinate-wise). Hence as $\bm x\notin T_{\bm t,\bm h}$, $a<0$ and as $\bm y \in T_{\bm t,\bm h}$,  $b\ge0$. Define $\bm z :=\alpha \bm x  + (1-\alpha) \bm y$ where $\alpha := b/(b-a) \in (0,1).$ Note that $\sum_{i=1}^d (z_i-t_i)/h_i = -1$ and $\bm x \leq \bm z \leq \bm y \leq \bm t$ which implies that $\bm z \in T_{\bm t,\bm h}$. Hence, $$f_{\bm t}(\bm y) \geq f_{\bm t}(\bm z) = f(\bm z) \geq f(\bm x) = f_t(\bm x),$$ 
		thereby yielding $f_{\bm t}(\bm y) \geq f_{\bm t}(\bm x)$. Here, the first inequality follows from the fact that $\bm z,\bm y \in T_{\bm t,\bm h}$, the third inequality follows from monotonicity of $f$ and the second and fourth equality follows from the fact that $\psi_{\bm t,\bm h}(\bm z)=\psi_{\bm t,\bm h}(\bm x)=0$ (cf.~\eqref{psi1defnt}). 
		
		Now let us look into the case where  $\bm x\in T_{\bm t,\bm h}$, $\bm y \notin T_{\bm t,\bm h}$ and $\bm x\leq \bm y.$ In this case $$f_{\bm t}(\bm y)=f(\bm y)\geq f(\bm x) \geq f(\bm x) - h_1(M_1^\star-\epsilon^\star)\psi_{\bm t,\bm h}(\bm x) = f_{\bm t}(\bm x).$$
		The only case left is when $\bm x\notin T_{\bm t,\bm h}$, $\bm y \notin T_{\bm t,\bm h}$ and $\bm x\leq \bm y.$ In this case $\psi_{\bm t,\bm h}(\bm x)=\psi_{\bm t,\bm h}(\bm y)=0$, hence the monotonicity of $f_{\bm t}$ directly follows from the monotonicity of $f$. This completes the proof of Lemma \ref{ftinF1}.
	\end{proof}
	
	\noindent Now we continue with the proof of Theorem~\ref{opconstant}. Let $[\ell,u]$ be any honest confidence band for the class $\mathcal{F}_1$. Let  
	$A$ be the event that $\{\ell(\bm x) \leq f_{\bm t}(\bm x) \mbox{   for all } \bm x\in[0,1]^d,  \mbox{   for some  } \bm t \in G\}.$
	Now, since $[\ell,u]$ is a confidence band for all $f\in \mathcal{F}_1$, and since $f_{\bm t} \in \mathcal{F}_1$, we have  $$\p_{f_{\bm t}}(A)\geq 1-\alpha \quad \mbox{   for all } \quad \bm t\in G.$$ Hence we have
	\begin{equation}\label{opt 6.27}
		\p_{f}\left(\|f- \ell\|_U \geq h_1(M_1^\star-\epsilon^\star)\right) \geq \p_{f}(A)\geq 1-\alpha-\min_{\bm t\in G}\left( \p_{f_{\bm t}}(A)-\p_{f}(A)\right).
	\end{equation}
	Here the first inequality follows from the fact that if $A$ happens then there exists $\bm t \in G\subset U$ such that $\ell \leq f_{\bm t}$, thereby yielding (as $\psi_{\bm t,\bm h}(\bm t) = 1$): $$\ell(\bm t) \leq f_{\bm t}(\bm t)=f(\bm t) - h_1(M_1^\star-\epsilon^\star).$$
	Hence it is enough to bound $\min_{\bm t\in G}\left( \p_{f_{\bm t}}(A)-\p_{f}(A)\right)$. Note that
	\begin{eqnarray}\label{optconsteq1}\min_{\bm t\in G} \p_{f_{\bm t}}(A)-\p_{f}(A) &\leq & |G|^{-1} \sum_{\bm t \in G} \left(\p_{f_{\bm t}}(A) - \p_{f}(A) \right) \nonumber \\ &=& |G|^{-1} \sum_{\bm t \in G}  \E_{f}\left( \left(\frac{d\p_{f_{\bm t}}}{d\p_{f}}(Y)-1\right) \mathbb{I}_A(Y)\right) \nonumber \\ &\leq &\E_{f}\Big||G|^{-1} \sum_{\bm t \in G} \left(\frac{d\p_{f_{\bm t}}}{d\p_{f}}(Y)-1\right) \Big|. \end{eqnarray}
	
	Now by Cameron-Martin-Girsanov's theorem, we have $$\log \frac{d\p_{f_{\bm t}}}{d\p_{f}}(Y)=n^{1/2}h_1(M_1^\star-\epsilon^\star) \sqrt{\Pi_{i=1}^d h_i} \|\psi\|X_{\bm t} - n(M_1^\star-\epsilon^\star)^2h_1^2(\Pi_{i=1}^d h_i) \|\psi\|^2/2$$
	where $$X_{\bm t}=(\Pi_{i=1}^n h_i)^{-1/2}\|\psi\|^{-1} \int \psi_{\bm t,\bm h} \,dW$$
	with $W$ being the standard Brownian sheet on $[0,1]^d$.  
	Note that here $X_{\bm t}$ follows a standard normal distribution and for $\bm t\neq \bm t^\prime \in G$, $X_{\bm t}$ and $X_{\bm t}^\prime$ are independent. Now let $$w_n:=n^{1/2}h_1(M_1^\star-\epsilon^\star) \sqrt{\Pi_{i=1}^d h_i} \|\psi\|.$$
	
	At this point, we need the following lemma (stated and proved in Lemma 6.2 in \cite{Dumbgen-Spokoiny-2001}).
	\begin{lem}\label{normal lemma}
		Let $\{Z_n\}_{n\ge 1}$ be a sequence of independent standard normal variables. If $v_m := (1-\epsilon_m)\sqrt{2\log m}$ with $\lim_{m \to \infty} \epsilon_m =0$ and $\lim_{m \to \infty} \epsilon_m \sqrt{\log m}=\infty$, then we have $$\lim_{m \to \infty} \E\left|\frac{1}{m}\sum_{i=1}^m \exp\left(v_mZ_i-\frac{v_m^2}{2}\right) - 1\right|=0.
		$$
	\end{lem}
	Define $\varepsilon_n := 1-(w_n/\sqrt{2\log |G|})$. If $\varepsilon_n \rightarrow 0$ and $\varepsilon_n \sqrt{\log |G|} \rightarrow \infty$ are satisfied, then by Lemma \ref{normal lemma} and \eqref{optconsteq1}, we have the following as $|G|\rightarrow \infty$: 
	\begin{equation}\label{minpft}
		\min_{t\in G} \p_{f_{\bm t}}(A) - \p_f(A) \rightarrow 0.
	\end{equation}
	Now let us choose
	\begin{equation}\label{eq:h_1}
		h_1 := (1-\epsilon_n)c\rho_n \quad \mbox{where} \quad \rho_n :=(\log(en)/n)^{1/(2+d)},
	\end{equation}
	with $\epsilon_n \to 0$ and $\epsilon_n\sqrt{\log n} \to \infty$ and $c$ is a constant to be chosen later. This implies that 
	$$\sqrt{2\log |G|}=(1- o(1))\sqrt{\frac{2d}{d+2}\log n}$$ 
	and for large $n$,
	$\sqrt{2\log |G|}<\sqrt{\frac{2d}{d+2}\log n}$.
	Hence,
	$$w_n=(1-\epsilon_n)^{(2+d)/2} c^{(2+d)/2} \|\psi^\ell\| \frac{(M_1^\star-\epsilon^\star)^{(2+d)/2}}{\prod_{i=1}^d (M_i^\star-\epsilon^\star)^{1/2}} \sqrt{\log(en)}.$$
	Let us now put:
	\begin{equation}\label{eq:c}
		c :=\frac{\prod_{i=1}^d (M_i^\star-\epsilon^\star)^{1/(2+d)}}{(M_1^\star-\epsilon^\star)} \left[\frac{(d+2)\|\psi^\ell\|^2}{2d}\right]^{-1/({2+d})}
	\end{equation}
	whence we have
	\begin{eqnarray*}
		\frac{w_n}{\sqrt{2 \log |G|}} & \sim &(1-\epsilon_n)^{(2+d)/2} c^{(2+d)/2} \|\psi^\ell\| \frac{(M_1^\star-\epsilon^\star)^{(2+d)/2}}{\Pi_{i=1}^d (M_i^\star-\epsilon^\star)^{1/2}} \sqrt{\frac{d+2}{2d}}\\
		&=& (1-\epsilon_n)^{(2+d)/2} \to 1 \mbox{  as  } n\to \infty.
	\end{eqnarray*}
	Also note that for large $n$, $(1-w_n/\sqrt{2\log |G|}) >0$ as  $\sqrt{2\log |G|}<\sqrt{\frac{2d}{d+2}\log n}$. Further, \begin{eqnarray*}
		\sqrt{\log |G|} \left(1-\frac{w_n}{\sqrt{2\log |G|}}\right) &\sim& \sqrt{\frac{d}{2+d}} \sqrt{\log n} (1- (1-\epsilon_n)^{1+d/2})\\
		& = & \sqrt{\frac{d}{2+d}} \left(\frac{2+d}{2}+o(1)\right)\epsilon_n \sqrt{\log n} \to \infty \quad  \mbox{(by assumption)}. \end{eqnarray*}
	Recall the definition $\Delta^{(\ell)} := \left((d+2)\|\psi^\ell\|^2/2d\right)^{-1/(2+d)}$ from the statement of the theorem. Hence, using \eqref{opt 6.27},~\eqref{minpft},~\eqref{eq:h_1}, and~\eqref{eq:c} we have:
	\begin{eqnarray*}
		1-\alpha & \leq &\liminf_{n \to \infty} \p_{f}\left(\|f- \ell\|_{U} \geq h_1(M_1^\star-\epsilon^\star)\right)\\
		& = &\liminf_{n \to \infty} \p_{f}\left(\|f- \ell\|_{U} \geq (1-\epsilon_n)\rho_n \Delta^{(\ell)} \Pi_{i=1}^d (M_i^\star-\epsilon^\star)^{1/(2+d)} \right)\\
		& \leq  & \liminf_{n \to \infty} \p_{f}\left(\|f- \ell\|_{U} \geq \gamma^\star \Pi_{i=1}^d (M_i^\star-\epsilon^\star)^{1/(2+d)}\rho_n \Delta^{(\ell)} \right) \qquad (\mbox{as } 1 - \epsilon_n \to 1 > \gamma^*)\\
		& \leq & \liminf_{n \to \infty} \p_{f}\left(\|f- \ell\|_{U} \geq \gamma^\star (L_1[\bm t_0]-\epsilon)^{d/(2+d)}\rho_n \Delta^{(\ell)}\right) \qquad (\mbox{by (ii) above})\\
		& =& \liminf_{n \to \infty} \p_{f}\left(\|f- \ell\|_{U} \geq \gamma  L_1^{d/(2+d)}[\bm t_0]\rho_n \Delta^{(\ell)}\right) \qquad (\mbox{by~\eqref{eq:L_1-t0}}).
	\end{eqnarray*}
	This completes the proof of part (a). \newline
	
	\noindent (b) We again restrict our attention to $(f- \hat{\ell})(\bm t_0)$. The other case can be done by a similar argument.
	For $i=1,2,\dots,d$, let us recall $$M_i^\star=\frac{\partial}{\partial x_i}f(\bm x)|_{\bm x= \bm t_0}.$$
	Fix $\epsilon>0$. As $f$ has continuous derivative on an open neighborhood $U$ of $\bm t_0$  we can find $\epsilon^\star>0$ and  a hyperrectangle $B_{\infty}(\bm t_0,\bm h_0)$ small enough such that 
	\begin{equation}\label{eq:Bd-partial-f}
		\sup_{\bm x \in B_{\infty}(\bm t_0,\bm h_0)} \frac{\partial}{\partial x_i}f(\bm x) \leq M_i^\star+\epsilon^\star
	\end{equation}
	and 
	\begin{equation}\label{eq:L_1-M_1}
		[\Pi_{i=1}^d (M_i^\star+\epsilon^\star)]^{1/d} \leq (1+\epsilon) {L}_1[\bm t_0].
	\end{equation}

	Recall that we have assumed without loss of generality $0< M_1^\star \leq M_2^\star \leq \ldots \leq M_d^\star$. Now let $\bm h=(h_1,\ldots,h_d)$ be such that $h_i:=\tilde{h} \times \frac{M_1^\star+\epsilon^\star}{M_i^\star+\epsilon^\star}$, where $\tilde h$ will be chosen later. Let \begin{equation}\label{eq:M}
		M:=\left[\Pi_{i=1}^d \frac{M_1^\star+\epsilon^\star}{M_i^\star+\epsilon^\star}\right]^{1/d}
	\end{equation}
	which implies that $\Pi_{i=1}^d h_i= \tilde{h}^d M^d$.
	Recall that $$ \hat{\ell}(\bm t_0)=\sup_{\bm h\in I : \bm t_0 \in A_{\bm h}} \left\{ \hat{f_{\bm h}}(\bm t_0) - \frac{\|\psi\|}{{\langle 1,\psi \rangle (n\Pi_{i=1}^dh_i)^{1/2}}} \left(\kappa_\alpha+\Gamma(2^d\Pi_{i=1}^dh_i)\right)\right\}$$ 
	and 
	\begin{eqnarray*}\hat{{f_{\bm h}}}(\bm t_0)& =  &\frac{1}{n^{1/2} (\Pi_{i=1}^dh_i) \langle 1,\psi\rangle} \int_{[0,1]^d} \psi_{\bm t_0,\bm h}(\bm x) \,dY(\bm x)\\
		&=&\frac{1}{\langle 1,\psi\rangle } \langle f(\bm t_0+\bm h\star \cdot),\psi(\cdot) \rangle + \frac{1}{n^{1/2} (\Pi_{i=1}^dh_i) \langle 1,\psi\rangle} \int_{B_\infty(\bm t_0,\bm h)} \psi_{\bm t_0,\bm h}(\bm x)\,dW(\bm x)
	\end{eqnarray*}
	where for $\bm h,\bm x\in\R^d$ we define $\bm h\star \bm x:= (h_1x_1,\ldots ,h_dx_d)$.

	Now, it follows from the definition of $\hat{\ell}({\bm t_0})$ that if $f(\bm t_0)-\hat{\ell}(\bm t_0) \geq (M_1^\star+\epsilon^\star) \tilde{h}$, then $$\hat{f}_{\bm h}(\bm t_0)-\frac{\|\psi\|\left(\kappa_{\alpha}+\Gamma(2^dM^d\tilde h^d)\right)}{n^{1/2}\tilde h^{d/2} M^{d/2}\langle 1,\psi \rangle} \leq f(\bm t_0)-(M_1^\star+\epsilon^\star)\tilde h, $$
	which can be rewritten as: 
	\begin{eqnarray}
		\frac{\int_{B_{\infty}(\bm t_0,\bm h)} \psi_{\bm t_0,\bm h}(\bm x)\, dW(\bm x)}{\|\psi\|\tilde h^{d/2} M^{d/2}} &\leq & -\frac{(n\tilde{h}^dM^d)^{1/2}}{\|\psi\|} \langle f(\bm t_0+\bm h\star\cdot) -f(\bm t_0)+(M_1^\star+\epsilon^\star)\tilde{h}, \psi(\cdot) \rangle \nonumber\\
		& & \qquad \qquad \qquad + \; \Gamma(2^dM^d\tilde{h}^d) + \kappa_\alpha.  \label{optconsteq19}
	\end{eqnarray}
	
	\begin{lem}\label{optlemma2}
		Suppose $\bm h\in(0,1/2]^d$ is such that $B_{\infty}(\bm t,\bm h) \subseteq B_{\infty}(\bm t_0,\bm h_0)$. Then  $$\langle f(\bm t+\bm h\star\cdot) -f(\bm t)+(M_1^\star+\epsilon^\star)\tilde{h}, \psi(\cdot) \rangle \geq (M_1^\star+\epsilon^\star)\tilde{h} \|\psi\|^2. $$
	\end{lem}
	\noindent {\it Proof of Lemma~\ref{optlemma2}:}
	Suppose $B_{\infty}(\bm t_0,\bm h) \subseteq B_\infty(\bm t_0,\bm h_0)$ holds.
	Recall that $$\psi(\bm x)= \left(1+\sum_{i=1}^d x_i\right) \mathbb{I}\left(\bm x \leq \boldsymbol{0},\sum_{i=1}^d x_i \geq -1\right).$$
	Let $\bm x \in [-1,0]^d$ be such that $\sum_{i=1}^d x_i \geq -1$. Then, by the mean value theorem and~\eqref{eq:Bd-partial-f},
	\begin{eqnarray*}
		f(\bm t+\bm h\star \bm x)-f(\bm t)& =&(\bm h\star \bm x)^\top \triangledown f(\boldsymbol{\xi})  \qquad (\mbox{for some} ~\boldsymbol{\xi} \in [\bm t, \bm t+ \bm h\star \bm x])\\
		& \geq &\sum_{i=1}^d h_ix_i(M_i^\star+\epsilon^\star)   \qquad (\mbox{note that } \bm h\star \bm x \leq \boldsymbol{0})\\
		&=& \tilde{h} (M_1^\star+\epsilon^\star) \sum_{i=1}^d x_i.
	\end{eqnarray*}
	Hence on the set $D:=\{\bm x \leq \boldsymbol{0},\sum_{i=1}^d x_i \geq -1\}$, we have
	$$f(\bm t+\bm h\star \bm x)-f(\bm t)+(M_1^\star+\epsilon^\star)\tilde{h} \geq \tilde{h}(M_1^\star+\epsilon^\star) \left(1+\sum_{i=1}^d x_i\right) \geq 0.$$
	Hence, as $\psi \ge 0$,
	\begin{equation}\label{optcondeq2}
		\langle f(\bm t+\bm h\star\cdot) -f(\bm t)+(M_1^\star+\epsilon^\star)\tilde{h}, \psi(\cdot) \rangle \geq \int_{D} \tilde{h}(M_1^\star+\epsilon^\star) \left(1+\sum_{i=1}^d x_i\right)^2 \, d{\bm x}= \tilde{h}(M_1^\star+\epsilon^\star)\|\psi\|^2.
	\end{equation} This completes the proof of Lemma \ref{optlemma2}. \qed
	\newline
	
	\noindent Let us get back to the proof of part (b) of Theorem~\ref{opconstant}. By \eqref{optconsteq19}, using \eqref{optcondeq2}, we get that as long as $B_{\infty}(\bm t_0,\bm h) \subseteq B_\infty(\bm t_0,\bm h_0)$, $(f-\hat{\ell})(\bm t_0) \geq (M_1^\star+\epsilon^\star) \tilde{h}$ implies that 
	$$\frac{\int_{B_{\infty}(\bm t_0,\bm h)} \psi_{\bm t_0,\bm h}(\bm x) \,dW(\bm x)}{\|\psi\|\tilde h^{d/2} M^{d/2}} \leq -\sqrt{n}\tilde{h}^{1+d/2} M^{d/2}(M_1^\star+\epsilon^\star) \|\psi\| + \Gamma(2^dM^d\tilde{h}^d) + \kappa_\alpha.$$
	Also note that 
	$$\frac{\int_{B_{\infty}(\bm t_0,\bm h)} \psi_{\bm t_0,\bm h}(\bm x) \,dW(\bm x)}{\|\psi\|\tilde h^{d/2} M^{d/2}} \sim N(0,1).$$
	Now let us choose $$\tilde{h} :=c(M_1^\star+\epsilon^\star)^{-\frac{2}{2+d}}\rho_n$$ where $\rho_n= (\log(en)/n)^{1/(2+d)}$ and the constant $c$ is to be chosen later.

	Note that $\rho_n \to 0$ as $n \to \infty$. Hence for large enough $n$, $B_{\infty}(\bm t_0,\bm h)\subseteq B_{\infty}(\bm t_0,\bm h_0)$ (here $\bm h$ depends on $n$). Also we have 
	\begin{eqnarray*}
		\Gamma(2^dM^d\tilde h^d)& \leq & \sqrt{\left(\frac{2d}{2+d}\right) \log (en)} \qquad \mbox{for large    } n, \\
		\sqrt{n}\tilde{h}^{1+d/2} M^{d/2}(M_1^\star+\epsilon^\star) \|\psi\|&=& M^{d/2}\|\psi\| c^{(d+2)/2} \sqrt{\log(en)}.
	\end{eqnarray*}
	Now let us pick $$c :=(1+\epsilon)\left( \frac{(d+2)\|\psi\|^2}{2d}\right)^{-1/(d+2)} M^{-\frac{d}{d+2}}.$$ Note that $\Delta^{(\ell)}=\left( \frac{(d+2)\|\psi\|^2}{2d}\right)^{-1/(d+2)}$ as defined in the statement of the theorem. Hence, for large $n$, we have (with $\Phi$ denoting the distribution function of standard normal distribution)
	\begin{eqnarray*}
		& &\p\left(f(\bm t_0)-\hat{\ell}(\bm t_0) \geq (M_1^\star+\epsilon^\star) \tilde{h}\right)\\ &\leq & \p\left(\frac{\int_{B_{\infty}(\bm t_0,\bm h)} \psi_{\bm t_0,\bm h}(\bm x)dW(\bm x)}{\|\psi\|\tilde h^{d/2} M^{d/2}} \leq -\sqrt{n}\tilde{h}^{1+d/2} M^{d/2}(M_1^\star+\epsilon^\star) \|\psi\| + \Gamma(2^dM^d\tilde{h}^d) + \kappa_\alpha \right)\\
		& =& \Phi\left( -\sqrt{n}\tilde{h}^{1+d/2} M^{d/2}(M_1^\star+\epsilon^\star) \|\psi\| + \Gamma(2^dM^d\tilde{h}^d) + \kappa_\alpha \right)\\
		&\leq & \Phi\left(\kappa_\alpha - \sqrt{\log(en)}\left[M^{d/2}\|\psi\| c^{(d+2)/2}- \sqrt{\frac{2d}{2+d}} \right] \right)\\
		&=& \Phi\left(\kappa_\alpha - \sqrt{\log(en)} \sqrt{\frac{2d}{2+d}} \left[(1+\epsilon)^{\frac{d+2}{2}}-1\right]\right) \to 0 \quad \mbox{   as   } n\to \infty.
	\end{eqnarray*}
	Hence,
	\begin{equation}
		\lim_{n\to \infty} \p\left( (f-\hat{\ell})(\bm t_0) \leq (M_1^\star+\epsilon^\star) \tilde{h}\right) = 1.\label{optconst3}\end{equation}
	Notice that,
	\begin{eqnarray}(M_1^\star+\epsilon^\star) \tilde{h}&=& (M_1^\star+\epsilon^\star) c (M_1^\star+\epsilon^\star)^{-\frac{2}{2+d}}\rho_n \nonumber\\
		&=&\rho_n (M_1^\star+\epsilon^\star)^{d/(2+d)}(1+\epsilon)\Delta^{(\ell)}M^{-d/(d+2)}\nonumber \\
		&=&(1+\epsilon)\Delta^{(\ell)}\rho_n \left(\frac{(M_1^\star+ \epsilon^\star)^d}{M^d}\right)^{1/(d+2)}\nonumber \\
		&=& (1+\epsilon)\Delta^{(\ell)}\rho_n  \left( \Pi_{i=1}^d (M_i^\star+\epsilon^\star) \right)^{1/(d+2)} \qquad \qquad (\mbox{using}~\eqref{eq:M})\nonumber \\
		&\leq & (1+\epsilon)\Delta^{(\ell)}\rho_n (1+\epsilon)^{d/(d+2)} L_1^{d/(d+2)}[\bm t_0]\nonumber \qquad \qquad (\mbox{using}~\eqref{eq:L_1-M_1}) \\
		&=&(1+\epsilon)^{(2d+2)/(d+2)}\Delta^{(\ell)} \rho_n L_1^{d/(d+2)}[\bm t_0].\nonumber
	\end{eqnarray}
	As $\epsilon >0$ is arbitrary, the above display along with~\eqref{optconst3} yields the desired result.\qed
	
	\subsection{Proof of Theorem \ref{opconstantcx}}\label{propconstantcx}
	(a)~ Once again, we prove only the bound for $\|f- {\ell}\|_{U}$ and the other case can be handled similarly. We will show that for any level $1-\alpha$ confidence band $[\ell,u]$ with guaranteed coverage probability for the class $\mathcal{F}_2$, and any $0<\gamma<1$, we have  $$\liminf_{n \to \infty}\p_{f}\left(\|f- \ell\|_{U} \geq \gamma \Delta^{(\ell)} L_2^{\frac{d}{4+d}}[\bm t_0] \rho_n\right) \geq 1-\alpha.$$ 
	
	We first introduce a rotation of the coordinate system, so that the Hessian of $f$ at $\bm t_0$ with respect to this new  coordinate system, is diagonal. If $\nabla^2 f(\bm t_0) = \bm P \bm D \bm P^\top$ is the spectral decomposition of the Hessian of $f$ at $\bm t_0$ (here $\bm P$ denotes an orthogonal matrix and $\bm D$ is a diagonal matrix), then for any point $\bm y \in \R^d$, we will define $\bm y' := \bm P^\top \bm y$, and for any set $S\subseteq \mathbb{R}^d$, we will define:
	\begin{equation}\label{Set-prime}
		S' := \{\bm P^\top \bm s: \bm s\in S\}~.
	\end{equation}
	Further, defining $$g(\bm t) := f(\bm P \bm t),$$ we note that $g(\bm t') = f(\bm P {\bm P}^\top \bm t) = f(\bm t)$ for all $\bm t$ (recall our notation that $\bm t' = \bm P^\top \bm t$), and $\nabla^2 g(\bm t_0') = \bm D$.

	Recall that by assumption, $f$ is twice continuously differentiable on an open neighborhood $U$ of $\bm t_0 \in (0,1)^d$ such that $L_2[\bm t_0] \equiv L_2[f,\bm t_0]:= \mathrm{det}(\nabla^2 f(\bm t_0))^{1/d} > 0$. Hence, $g$ is twice continuously differentiable on the open neighborhood $U'$ of $\bm t_0'$. Denote the $i^{\mathrm{th}}$ smallest eigenvalue of $\nabla^2 f(\bm t_0)$ by $\lambda_i(\bm t_0)$, for $i=1,\ldots d$.  Since $\gamma<1$, we can find $\epsilon>0$ and  $\gamma^*<1$ such that:
	\begin{equation}\label{eq:L_2-t_0}
		\gamma  L_2^{d/(4+d)}[\bm t_0]=\gamma^\star( L_2[\bm t_0]-\epsilon)^{d/(4+d)}.
	\end{equation} 
	Using the twice continuous differentiability of $g$ on $U'$, we can find $\bm h_0\in [0,1]^d$ and  $\epsilon^\star >0$ small enough such that \begin{itemize} 
		\item[(i)] $B_\infty'(\bm t_0',\bm h_0) \subset U'$, 
		
		\item[(ii)] $\left[\prod_{i=1}^d (\lambda_i(\bm t_0)-\epsilon^\star)\right]^{1/d} \geq {L}_2[\bm t_0]-\epsilon$,
		
		\item[(iii)] for all $\bm x' \in B_\infty'(\bm t_0',\bm h_0)$, we have $$\sup_{\bm v \in B_2(\boldsymbol{0},1)}\left|\bm v^\top \left(\nabla^2 g(\bm x') - \nabla^2 g(\bm t_0')\right)\bm v\right| < \epsilon^*$$
		where $B_2(\boldsymbol{0},1)$ denotes the closed ball around $\boldsymbol{0} \in \R^d$ with radius $1$, and $B_{\infty}'(\bm y, \bm h) := (B_{\infty}(\bm P \bm y, \bm h))'$; recall the notation in~\eqref{Set-prime}.
	\end{itemize}
	
	Next, let $\bm h \equiv (h_1,h_2,\ldots,h_d)$ be such that $h_1 \in (0,1/2]$ and \begin{equation}\label{convex_hi}
		h_i= h_1\times \sqrt{\frac{\lambda_1(\bm t_0)-\epsilon^\star}{\lambda_i(\bm t_0)-\epsilon^\star}}, 
		\qquad \qquad \mbox{for } i=2,\ldots, d.   
	\end{equation}    Let us define a set of grid points $G \subset [0,1]^d$ for bandwidth $\bm h$ as 
	$$G :=\left\{\bm t=(t_1,\ldots,t_d) \in [0,1]^d: t_i=t_{0i}+\alpha_d h_i(2k_i) \mbox{ for some integer } k_i  , B_\infty(\bm t, \alpha_d \bm h) \subset B_{\infty}(\bm t_0,\bm h_0) \right\}$$
	where $\bm t_0=(t_{01},\ldots,t_{0d})$ and $\alpha_d := \sqrt{2(d+3)/(d+1)}$. For $\bm t'\in G'$, let 
	\begin{equation}\label{eq:g_t'}
		g_{\bm t'}(\bm x):= g(\bm x)-h_1^2(\lambda_1(\bm t_0)-\epsilon^\star) \psi^*\left(\frac{x_1-t_1'}{h_1},\ldots,\frac{x_d-t_d'}{h_d}\right)
	\end{equation}
	where 
	\begin{equation}\label{eq:psi_*}
		\psi^*(\cdot):= (G_A - G_0)(\bm P \cdot)
	\end{equation}
	with $G_A$, $G_0$ defined as follows:
	$$G_A(\bm y) := \frac{\|\bm y\|^2}{2} \mathbbm{1}_{\|\bm y\| \le \sqrt{2(d+3)/(d+1)}}\quad \text{and}\quad G_0(\bm y) := \left(-1+ \frac{\sqrt{2}(d+2)}{\sqrt{(d+1)(d+3)}} \|\bm y\|\right) \mathbbm{1}_{\|\bm y\| \le \sqrt{2(d+3)/(d+1)}}~.$$
	Note that $\psi_2^\ell(\bm x) = (G_A - G_0)\left(\sqrt{2(d+3)/(d+1)} \bm x\right)$.
	
	We will now show that for every $\bm t' \in G'$, the function $g_{\bm t'} \in \mathcal{F}_2$.
	
	\begin{lem}\label{ftinF2}
		$g_{\bm t'} \in \mathcal{F}_2 \mbox{  for all   } \bm t' \in G'.$
	\end{lem}
	\begin{proof}
		Fix $\bm t' \in G'$ and a vector $\bm v \in B_d(\boldsymbol{0},1)$. In order to prove Lemma \ref{ftinF2}, it suffices to show that the univariate function $h:\mathbb{R}\to \mathbb{R}$ defined as $h(s) := g_{\bm t'} (s \bm v)$ is convex. To prove this, take scalars $\alpha > \beta$, such that $\alpha \bm v$ and $\beta \bm v \in B_\infty'(\bm t_0', \bm h_0)$, and define $$\phi(\bm x) \equiv \phi_{\bm t', \bm h}(\bm x) := \psi^*\left(\frac{x_1-t_1'}{h_1},\ldots,\frac{x_d-t_d'}{h_d}\right), \qquad \mbox{for } \; \bm x \in [0,1]^d.$$
		Then, we have (using~\eqref{eq:g_t'})  
		\begin{eqnarray}
			&&h'(\alpha) - h'(\beta)\nonumber\\ &=& \left(\nabla g(\alpha \bm v) - \nabla g(\beta \bm v)\right)^\top \bm v -  h_1^2(\lambda_1(\bm t_0)-\epsilon^\star)\left[\nabla \phi(\alpha \bm v) - \nabla \phi(\beta \bm v)\right]^\top \bm v \nonumber \\&=& (\alpha-\beta) \left[\bm v^\top \nabla^2 g(\xi\bm v) \bm v - h_1^2(\lambda_1(\bm t_0)-\epsilon^\star) \bm v^\top \nabla^2 \phi (\eta \bm v) \bm v\right] \label{eq:hprime}
		\end{eqnarray}
		for some $\xi, \eta$ lying between $\alpha$ and $\beta$. First, note that since $\xi \bm v \in B_{\infty}'(\bm t_0', \bm h_0)$, we have:
		\begin{equation}\label{eq:nabla_g}
			\bm v^\top \nabla^2 g(\xi\bm v) \bm v  \ge \left[\bm v^\top \nabla^2 g(\bm t_0') \bm v - \epsilon^\star\right] =  \sum_{i=1}^d (\lambda_i(\bm t_0)-\epsilon^\star) v_i^2~.
		\end{equation}
		Here the inequality above follows from (iii) above and the equality follows from the fact that $\nabla^2 g(\bm t_0') $ is diagonal. Denote $D_{\bm h^{-1}}$ to be the diagonal matrix with diagonal entries $(1/h_i)_{i=1}^d$.
		Next, let us try to study the term $\bm v^\top \nabla^2 \phi (\eta \bm v) \bm v$. Note that  
		\begin{eqnarray}
			\bm v^\top \nabla^2 \phi (\eta \bm v) \bm v &=& \bm v^\top D_{\bm h^{-1}}^\top \left[\nabla^2 \psi^* ((\eta \bm v- \bm t')/\bm h)\right] D_{\bm h^{-1}} \bm v  \nonumber \\ 
			&=& \bm v^\top D_{\bm h^{-1}}^\top {\bm P}^\top \left[\nabla^2 (G_A-G_0)({\bm P}((\eta \bm v- \bm t')/\bm h))\right] {\bm P} D_{\bm h^{-1}} \bm v   \nonumber \\
			& \le & \bm v^\top D_{\bm h^{-1}}^\top  D_{\bm h^{-1}} \bm v \;\; = \sum_{i=1}^d \frac{v_i^2}{h_i^2}. \label{eq:Nabla_phi}
		\end{eqnarray} 
		Here, the inequality above follows from the fact as long as $\|\bm y\| < \sqrt{2(d+3)/(d+1)}$, $\nabla^2 G_A({\bm y})$ is  the identity matrix which yields the upper bound. Note that $\nabla^2 G_0({\bm y})$ is a positive semidefinite matrix (as $G_0$ is a convex function) and hence we can ignore the corresponding term.

		Hence, by using \eqref{eq:hprime}, \eqref{eq:nabla_g} and \eqref{eq:Nabla_phi} we have:
		\begin{eqnarray*}
			h'(\alpha) - h'(\beta)  &\ge& (\alpha-\beta) \left[\sum_{i=1}^d (\lambda_i(\bm t_0)-\epsilon^\star) v_i^2 - (\lambda_1(\bm t_0)-\epsilon^\star) \sum_{i=1}^d (h_1/h_i)^2 v_i^2  \right] = 0,
		\end{eqnarray*}
		where the equality follows from \eqref{convex_hi}.
		This completes the proof of Lemma \ref{ftinF2}.
	\end{proof}
	
	\noindent Let us get back to the proof of Theorem~\ref{opconstantcx}. Recall that $[\ell, u]$ is any given confidence band with guaranteed coverage. Let us define $A$
	to be the event that $$\{\ell(\bm P \bm x) \leq g_{\bm t'}(\bm x) \mbox{   for all } \bm x\in ([0,1]^d)',  \mbox{   for some  } \bm t' \in G'\}.$$  Now, since $[\ell,u]$ is a confidence band for all functions in $\mathcal{F}_2$, and since the function $g_{\bm t'}(\bm P^\top \cdot) \in \mathcal{F}_2$, we have  $$\p_{g_{\bm t'}}(A)\geq 1-\alpha \quad \mbox{   for all } \;\;\bm t'\in G'.$$ 
	Hence, defining $\ell^*(\bm x) := \ell (\bm P \bm x)$, we have:
	\begin{equation}\label{opt 6.279}
		\p_{g}\left(\|g- \ell^*\|_{U'} \geq h_1^2(\lambda_1(\bm t_0)-\epsilon^\star)\right) \geq \p_{g}(A)\geq 1-\alpha-\min_{\bm t'\in G'}\left( \p_{g_{\bm t'}}(A)-\p_{g}(A)\right).
	\end{equation}
	Note that the first inequality follows from the fact that if $A$ happens then there exists $\bm t' \in G'\subset U'$ such that $\ell^*(\bm x) \leq g_{\bm t'}(\bm x)$ on $([0,1]^d)'$, thereby yielding (as $\psi^*({\bm 0}) = 1$; see~\eqref{eq:psi_*}): $$\ell^*(\bm t') \leq g_{\bm t'}(\bm t')=g(\bm t') - h_1^2(\lambda_1(\bm t_0)-\epsilon^\star).$$ 
	Hence it is enough to bound $\min_{\bm t'\in G'}\left( \p_{g_{\bm t'}}(A)-\p_{g}(A)\right)$. Observe that
	\begin{eqnarray}\label{6.289}\min_{\bm t'\in G'} \p_{g_{\bm t'}}(A)-\p_{g}(A) &\leq & |G'|^{-1} \sum_{\bm t' \in G'} \left(\p_{g_{\bm t'}}(A) - \p_{g}(A) \right) \nonumber \\ &=& |G'|^{-1} \sum_{\bm t' \in G'}  \E_{g}\left( \left(\frac{d\p_{g_{\bm t'}}}{d\p_{g}}(Y)-1\right) \mathbb{I}_A(Y)\right) \nonumber \\ &\leq &\E_{g}\Big||G'|^{-1} \sum_{\bm t' \in G'} \left(\frac{d\p_{g_{\bm t'}}}{d\p_{g}}(Y)-1\right) \Big|. \end{eqnarray}
	
	Now by Cameron-Martin-Girsanov's theorem, we have $$\log \frac{d\p_{g_{\bm t'}}}{d\p_{g}}(Y)=n^{1/2}h_1^2(\lambda_1(\bm t_0)-\epsilon^\star) \sqrt{\Pi_{i=1}^d h_i} \|\psi^*\|X_{\bm t'} - \frac{n}{2}(\lambda_1(\bm t_0)-\epsilon^\star)^2 h_1^4 (\Pi_{i=1}^d h_i) \|\psi^*\|^2$$
	where $$X_{\bm t'} :=(\Pi_{i=1}^d h_i)^{-1/2}\|\psi^*\|^{-1} \int \phi_{\bm t', \bm h} ~dW$$
	with $W$ being the standard Brownian sheet on $[0,1]^d$. 
	Note that $X_{\bm t'}$ follows a standard normal distribution and for $\bm s'\neq \bm t^\prime \in G'$, $X_{\bm s'}$ and $X_{\bm t'}$ are independent. Now let $$w_n:=n^{1/2}h_1^2(\lambda_1(\bm t_0)-\epsilon^\star) \sqrt{\Pi_{i=1}^d h_i} \|\psi^*\| \qquad \text{and}\qquad \varepsilon_n := 1-(w_n/\sqrt{2\log |G'|}).$$
	If $\varepsilon_n \rightarrow 0$ and $\varepsilon_n \sqrt{\log |G'|} \rightarrow \infty$ are satisfied, then by Lemma \ref{normal lemma} and \eqref{6.289}, we have the following as $|G'|\rightarrow \infty$: 
	\begin{equation}\label{minpft9}
		\min_{\bm t'\in G'} \p_{g_{\bm t'}}(A) - \p_g(A) \rightarrow 0.
	\end{equation}
	Now let us choose \begin{equation}\label{eq:h_conv}
		h_1 := \sqrt{(1-\epsilon_n)c\rho_n} \qquad  \mbox{where} \qquad \rho_n :=(\log(en)/n)^{2/(4+d)},
	\end{equation}  with $\epsilon_n \to 0$ and $\epsilon_n\sqrt{\log n} \to \infty$ and $c$ is a constant to be chosen later. This implies that 
	$$\sqrt{2\log |G'|}=(1- o(1))\sqrt{\frac{2d}{d+4}\log n}$$ 
	and for large $n$,
	$\sqrt{2\log |G'|}<\sqrt{\frac{2d}{d+4}\log n}$. Hence,
	$$w_n=[c(1-\epsilon_n)]^{(4+d)/4} \|\psi^*\| \frac{(\lambda_1(\bm t_0)-\epsilon^\star)^{(4+d)/4}}{\prod_{i=1}^d (\lambda_i(\bm t_0)-\epsilon^\star)^{1/4}} \sqrt{\log(en)}.$$
	Let us now put: \begin{equation}\label{eq:c_convex}
		c:=\frac{\prod_{i=1}^d (\lambda_i(\bm t_0)-\epsilon^\star)^{1/(4+d)}}{(\lambda_1(\bm t_0)-\epsilon^\star)} \left[\frac{(d+4)\|\psi^\star\|^2}{2d}\right]^{-2/(4+d)}   
	\end{equation}whence we have:
	\begin{eqnarray*}
		\frac{w_n}{\sqrt{2 \log |G'|}} = (1-\epsilon_n)^{(4+d)/4} \to 1 \mbox{  as  } n\to \infty.\\ 
	\end{eqnarray*}
	Also note that for large $n$, $(1-w_n/\sqrt{2\log |G|}) >0$ as  $\sqrt{2\log |G|}<\sqrt{\frac{2d}{d+2}\log n}$. Also, note that: \begin{eqnarray*}\sqrt{\log |G'|} \left(1-\frac{w_n}{\sqrt{2\log |G'|}}\right) &\sim& \sqrt{\frac{d}{4+d}} \sqrt{\log n} (1- (1-\epsilon_n)^{1+d/4})\\
		& = & \sqrt{\frac{d}{4+d}} \left(\frac{4+d}{4}+o(1)\right)\epsilon_n \sqrt{\log n} \to \infty \quad  \mbox{(by assumption)}. \end{eqnarray*}
	
	Let $\Lambda_*^{(\ell)} := \left((d+4)\|\psi^*\|^2/2d\right)^{-2/(4+d)}$. 
	Hence, by \eqref{opt 6.279}, \eqref{minpft9}, \eqref{eq:h_conv} and \eqref{eq:c_convex}, we have:
	\begin{eqnarray*}
		1-\alpha & \leq &\liminf_{n \to \infty} \p_{g}\left(\|g- \ell^*\|_{U'} \geq h_1^2(\lambda_1(\bm t_0) -\epsilon^\star)\right)\\
		& = &\liminf_{n \to \infty} \p_{g}\left(\|g- \ell^*\|_{U'} \geq (1-\epsilon_n)\rho_n \Lambda_\star^{(\ell)} \Pi_{i=1}^d (\lambda_i(\bm t_0)-\epsilon^\star)^{1/(4+d)} \right)\\
		& \leq  & \liminf_{n \to \infty} \p_{g}\left(\|g- \ell^*\|_{U'} \geq \gamma^\star\rho_n \Lambda_\star^{(\ell)} \Pi_{i=1}^d (\lambda_i(\bm t_0)-\epsilon^\star)^{1/(4+d)} \right) \qquad (\mbox{as } 1-\epsilon_n \to 1 > \gamma^\star) \\
		& \leq & \liminf_{n \to \infty} \p_{g}\left(\|g- \ell^*\|_{U'} \geq \gamma^\star (L_2[\bm t_0]-\epsilon)^{d/(4+d)}\rho_n \Lambda_\star^{(\ell)}\right) \qquad (\mbox{by  (ii) above}) \\
		& =& \liminf_{n \to \infty} \p_{g}\left(\|g- \ell^*\|_{U'} \geq \gamma  L_2^{d/(4+d)}[\bm t_0]\rho_n \Lambda_\star^{(\ell)}\right) \qquad (\mbox{by } \eqref{eq:L_2-t_0}).
	\end{eqnarray*}
	Now, note that $\|g-\ell^*\|_{U'} = \|f-\ell\|_U$ and $\|\psi^*\|^2 =(\sqrt{2(d+3)/(d+1)})^d \|\psi^\ell\|^2$. 
	For bounding $\|u-f\|_U$, one works with the transformed kernel (instead of $\psi^*$ in~\eqref{eq:g_t'}) $$\xi_2^u(\bm x) := \left(1-\frac{\|\bm x\|^2}{2}\right)\mathbbm{1}_{\|\bm x\| \le \sqrt{2}}$$
	which is related to the kernel function $\psi_2^u$ by the relation $\psi_2^u(\bm x) := \xi_2^u(\sqrt{2}\bm x)$.
	This introduces the $(\sqrt{2})^d$ term in $\Delta^{(u)}$ and completes the proof of part (a) of Theorem \ref{opconstantcx}. \newline

	\noindent (b) As in the proof of Theorem \ref{opconstant} (b), we restrict our attention to $(f- \hat{\ell})(\bm t_0)$, since the other case can be handled by a similar argument. Fix $\epsilon>0$. For notational convenience, we will abbreviate $\psi^\ell$ by $\psi$. Since $f$ is twice continuously differentiable on an open neighborhood $U$ of $\bm t_0$,  we can find $\epsilon^\star>0$ and  a hyperrectangle $B_{\infty}(\bm t_0,\bm h_0)$ small enough such that, for all $1\le i \le d$, 
	\begin{equation*}\label{HessianUB1}
		\sup_{\bm x \in B_{\infty}(\bm t_0,\bm h_0)} H_{ii}(\bm x) \leq H_{ii}(\bm t_0)+\epsilon^\star
	\end{equation*}
	and 
	\begin{equation}\label{HessianUB2}
		[\Pi_{i=1}^d (H_{ii}(\bm t_0)+\epsilon^\star)]^{1/d} \leq (1+\epsilon) {L}_{2,\star}[f,\bm t_0].
	\end{equation}
	Next, let $\bm h=(h_1,\ldots,h_d)$ be such that \begin{equation}\label{eq:h_i-cvx}
		h_i:=\tilde{h} \times \sqrt{\frac{H_{11}(\bm t_0)+\epsilon^\star}{H_{ii}(\bm t_0)+\epsilon^\star}}, 
	\end{equation} 
	where $\tilde{h}$ will be chosen later. Let \begin{equation}\label{Convex_h_M}
		M:=\left[\Pi_{i=1}^d \frac{H_{11}(\bm t_0)+\epsilon^\star}{H_{ii}(\bm t_0)+\epsilon^\star}\right]^{1/2d}
	\end{equation} 
	which implies that $\Pi_{i=1}^d h_i= \tilde{h}^d M^d$.
	Next, as before, recall that $$ \hat{\ell}(\bm t_0)=\sup_{\bm h\in I : \bm t_0 \in A_{\bm h}} \left\{ \hat{f_{\bm h}}(\bm t_0) - \frac{\|\psi\|}{{\langle 1,\psi \rangle (n\Pi_{i=1}^dh_i)^{1/2}}} \left(\kappa_\alpha+\Gamma(2^d\Pi_{i=1}^dh_i)\right)\right\}$$ 
	and  
	\begin{eqnarray*}\hat{{f_{\bm h}}}(\bm t_0)& =  &\frac{1}{n^{1/2} (\Pi_{i=1}^dh_i) \langle 1,\psi\rangle} \int_{[0,1]^d} \psi_{\bm t_0,\bm h}(\bm x)\,dY(\bm x)\\
		&=&\frac{1}{\langle 1,\psi\rangle } \langle f(\bm t_0+\bm h\star \cdot),\psi(\cdot) \rangle + \frac{1}{n^{1/2} (\Pi_{i=1}^dh_i) \langle 1,\psi\rangle} \int_{B_\infty(\bm t_0,\bm h)} \psi_{\bm t_0,\bm h}(\bm x)\,dW(\bm x)
	\end{eqnarray*}
	where for $\bm h,\bm x\in\R^d$ we define $\bm h\star \bm x:= (h_1x_1,\ldots ,h_dx_d)$.

	Recall that $\alpha_d := \sqrt{2(d+3)/(d+1)}$. Now, it follows from the definition of $\hat{\ell}({\bm t_0})$ that if $f(\bm t_0)-\hat{\ell}(\bm t_0) \geq (H_{11}(\bm t_0)+\epsilon^\star) d\tilde{h}^2/\alpha_d^2$, then $$\hat{f}_{\bm h}(\bm t_0)-\frac{\|\psi\|\left(\kappa_{\alpha}+\Gamma(2^dM^d\tilde h^d)\right)}{n^{1/2}\tilde h^{d/2} M^{d/2}\langle 1,\psi \rangle} \leq f(\bm t_0)-(H_{11}(\bm t_0)+\epsilon^\star)d\tilde h^2/\alpha_d^2, $$
	which can be rewritten as: 
	
	\begin{eqnarray}
		\frac{\int_{B_{\infty}(\bm t_0,\bm h)} \psi_{\bm t_0,\bm h}(\bm x)dW(\bm x)}{\|\psi\|\tilde h^{d/2} M^{d/2}} &\leq & -\frac{(n\tilde{h}^dM^d)^{1/2}}{\|\psi\|} \left \langle f(\bm t_0+\bm h\star\cdot) -f(\bm t_0)+(H_{11}(\bm t_0) +\epsilon^\star)d\tilde{h}^2/\alpha_d^2, \psi(\cdot) \right \rangle \nonumber\\
		& & \qquad \qquad \qquad + \Gamma(2^dM^d\tilde{h}^d) + \kappa_\alpha.  \label{optconsteq199}
	\end{eqnarray}
	
	\begin{lem}\label{optlemma29}
		Suppose $\bm h\in[0,1]^d$ is such that $B_{\infty}(\bm t,\bm h) \subseteq B_{\infty}(\bm t_0,\bm h_0)$. Then  $$\langle f(\bm t+\bm h\star\cdot) -f(\bm t)+(H_{11}(\bm t_0)+\epsilon^\star)d\tilde{h}^2/\alpha_d^2, \psi(\cdot) \rangle \geq \textcolor{black}{(H_{11}(\bm t_0)+\epsilon^\star)d\tilde{h}^2 \|\psi\|^2}/\alpha_d^2~.$$
	\end{lem}
	\begin{proof}
		Observe that, using~\eqref{eq:h_i-cvx},
		\begin{eqnarray}\label{scale7}
			&& \langle f(\bm t+\bm h\star\cdot) -f(\bm t)+\textcolor{black}{(H_{11}(\bm t_0)+\epsilon^\star)d\tilde{h}^2/\alpha_d^2}, \psi(\cdot) \rangle\nonumber\\&=& \frac{(H_{11}(\bm t_0))+\epsilon^\star)d\tilde{h}^2}{\alpha_d^2} \left\langle \frac{f(\bm t+\bm h\star\cdot) -f(\bm t)}{\textcolor{black}{(H_{ii}(\bm t_0)+\epsilon^\star)d h_i^2/\alpha_d^2}} 
			+ 1~,~ \psi(\cdot)\right\rangle
		\end{eqnarray}
		for every $1\le i \le d$. Note that the Hessian of the function $f(\bm t+\bm h\star\cdot) -f(\bm t)$ is given by $\mathcal{H} \star \nabla^2 f(\bm t + \bm h\star \cdot)$, where $\mathcal{H} = ((\mathcal{H}_{i,j}))$ with $\mathcal{H}_{i,j} := h_ih_j$ and $\star$ denotes elementwise product of matrices. As the largest entry of a nonnegative definite matrix is always on its diagonal (by Lemma \ref{nndrs}), we next claim that with
		$$v:= \arg \max_{1\le i\le d} (H_{ii}(\bm t_0)+\epsilon^\star)h_i^2$$
		the function $g(\cdot) := (f(\bm t+\bm h\star\cdot) -f(\bm t))/(H_{vv}(\bm t_0)+\epsilon^*)h_v^2$ is in $\mathbb{H}_{2,1}'$, where we define the superclass $\mathbb{H}_{\beta,L}'\supseteq \mathbb{H}_{\beta,L}$ as the set of all functions $g: [0,1]^d \to \mathbb{R}$ satisfying \eqref{seccondhld} only. To see this, note that, for any ${\bm y} = (y_1,\ldots, y_d)$ and $ {\bm z} = (z_1,\ldots, z_d) \in \R^d$, we have the following by a telescopic argument: 
		\begin{eqnarray}
			\sum_{j=1}^d\left|\frac{\partial g(\bm y)}{\partial x_j}- \frac{\partial g(\bm z)}{\partial x_j}\right| &\le& \sum_{j=1}^d \sum_{i=1}^d    \left|\frac{\partial}{\partial x_j} g(y_1,\ldots,y_{i-1},z_i,\ldots,z_d) -  \frac{\partial}{\partial x_j} g(y_1,\ldots,y_{i},z_{i+1},\ldots,z_d)\right| \nonumber \\ 
			&\le& \|\bm y-\bm z\|_1\sum_{j=1}^d\sup_{\bm u \in L(\bm y,\bm z)}\max_{1\le i\le d}\left|\frac{\partial^2}{\partial x_i \partial x_j} g(\bm u)\right|\nonumber\\ &\le& d\|\bm y-\bm z\|_1\sup_{\bm u \in L(\bm y,\bm z)}\max_{1\le i,j\le d}\left|\frac{\partial^2}{\partial x_i \partial x_j} g(\bm u)\right| \label{eqsyst7}
		\end{eqnarray}
		where $L(\bm y,\bm z)$ denotes the hyperrectangle defined by the two extreme points $\bm y$ and $\bm z$. Next we will use the following lemma (proved in Appendix~\ref{techres}).
		\begin{lem}\label{techresl}
			The function $\psi_2^\ell$ defined in \eqref{psi2defnt} satisfies:
			$$
			\langle s,\psi_2^\ell\rangle \ge \|\psi_2^\ell\|^2 - \langle 1, \psi_2^\ell \rangle$$ 
			for all $s: B_d\to \R$ whenever $s(\boldsymbol{0}) \ge 0$ and the function $\bm x \to s(\bm x/\alpha_d)$ with support $\{\bm y: \|\bm y\| \le \alpha_d\}$ is in $H_{2,1}'$.
		\end{lem}
		We will take $s(\cdot) := \frac{f(\bm t+\bm h\star\cdot) -f(\bm t)}{(H_{vv}(\bm t_0)+\epsilon^\star)d h_v^2/\alpha_d^2}$ in the above lemma. By~\eqref{eqsyst7} we see that $s(\cdot/\alpha_d) \in \mathbb{H}_{2,1}'$ and $s({\bm 0}) = 0$. Thus, Lemma \ref{optlemma29} now follows from Lemma \ref{techresl}.
	\end{proof}
	\vspace{0.1in}
	
	\noindent We now continue with the proof of Theorem~\ref{opconstantcx} (b). By \eqref{optconsteq199} and Lemma \ref{optlemma29} we get that as long as $B_{\infty}(\bm t_0,\bm h) \subseteq B_\infty(\bm t_0,\bm h_0)$, $(f-\hat{\ell})(\bm t_0) \geq (H_{11}(\bm t_0)+\epsilon^\star) d\tilde{h}^2/\alpha_d^2$ implies that 
	$$\frac{\int_{B_{\infty}(\bm t_0,\bm h)} \psi_{\bm t_0,\bm h}(\bm x)dW(\bm x)}{\|\psi\|\tilde h^{d/2} M^{d/2}} \leq -\sqrt{n}\tilde{h}^{2+d/2} M^{d/2}(H_{11}(\bm t_0)+\epsilon^\star) d\|\psi\|/\alpha_d^2 + \Gamma(2^dM^d\tilde{h}^d) + \kappa_\alpha.$$
	Also note that 
	$$\frac{\int_{B_{\infty}(\bm t_0,\bm h)} \psi_{\bm t_0,\bm h}(\bm x) \, dW(\bm x)}{\|\psi\|\tilde h^{d/2} M^{d/2}} \sim N(0,1).$$
	Now let us choose $$\tilde{h} :=\sqrt{c(H_{11}(\bm t_0)+\epsilon^\star)^{-\frac{4}{4+d}}\rho_n}$$ where $\rho_n :=(\log(en)/n)^{2/(4+d)}$ and the constant $c$ will be chosen later.

	Note that $\rho_n \to 0$ as $n \to \infty$. Hence for large enough $n$, $B_{\infty}(\bm t_0,\bm h)\subseteq B_{\infty}(\bm t_0,\bm h_0)$ (here $\bm h$ depends on $n$). Also we have 
	\begin{eqnarray*}
		\Gamma(2^dM^d\tilde h^d)& \leq & \textcolor{black}{\sqrt{\left(\frac{2d}{4+d}\right) \log n}} \qquad \mbox{for large    } n\\
		\sqrt{n}\tilde{h}^{2+d/2} M^{d/2}(H_{11}(\bm t_0)+\epsilon^\star) \|\psi\|&=& M^{d/2}\|\psi\| c^{(d+4)/4} \sqrt{\log(en)}.
	\end{eqnarray*}
	Now let us pick \begin{equation}\label{convex_c}
		c :=(1+\epsilon)\left( \frac{(d+4)\|\psi\|^2}{2d}\right)^{-2/(d+4)} M^{-\frac{2d}{d+4}}~\alpha_d^\frac{8}{4+d} d^{-\frac{4}{d+4}}.
	\end{equation}
	Define $\Lambda^{(\ell)} :=\left( \frac{(d+4)\|\psi\|^2}{2d}\right)^{-2/(d+4)}$. Hence, for large $n$, we have \begin{eqnarray*}
		& &\p\left(f(\bm t_0)-\hat{\ell}(\bm t_0) \geq (H_{11}(\bm t_0)+\epsilon^\star) d\tilde{h}^2/\alpha_d^2\right)\\ &\leq & \p\left(\frac{\int_{B_{\infty}(\bm t_0,\bm h)} \psi_{\bm t_0,\bm h}(\bm x)dW(\bm x)}{\|\psi\|\tilde h^{d/2} M^{d/2}} \leq  -\sqrt{n}\tilde{h}^{2+d/2} M^{d/2}(H_{11}(\bm t_0)+\epsilon^\star) d \|\psi\|/\alpha_d^2 + \Gamma(2^dM^d\tilde{h}^d) + \kappa_\alpha\right)\\
		& =& \Phi\left( -\sqrt{n}\tilde{h}^{2+d/2} M^{d/2}(H_{11}(\bm t_0)+\epsilon^\star) d\|\psi\|/\alpha_d^2 + \Gamma(2^dM^d\tilde{h}^d) + \kappa_\alpha \right)\\
		&\leq & \Phi\left(\kappa_\alpha - \sqrt{\log(en)}\left[M^{d/2}d\|\psi\| c^{(d+4)/4}/\alpha_d^2- \sqrt{\frac{2d}{4+d}} \right] \right)\\
		&=& \Phi\left(\kappa_\alpha - \sqrt{\log(en)} \sqrt{\frac{2d}{4+d}} \left[(1+\epsilon)^{\frac{d+4}{4}}-1\right]\right) \to 0 \;\;\mbox{   as   } \;\;n\to \infty.
	\end{eqnarray*}
	Thus, 
	\begin{equation}
		\lim_{n\to \infty} \p\left( (f-\hat{\ell})(\bm t_0) \leq (H_{11}(\bm t_0)+\epsilon^\star) d\tilde{h}^2/\alpha_d^2\right) = 1.\label{optconst39}\end{equation}
	Now, \begin{eqnarray}&&(H_{11}(\bm t_0)+\epsilon^\star) d\tilde{h}^2/\alpha_d^2\nonumber\\ &=& d(H_{11}(\bm t_0)+\epsilon^\star) c (H_{11}(\bm t_0)+\epsilon^\star)^{-\frac{4}{4+d}}\rho_n/\alpha_d^2 \nonumber\\
		&= &\rho_n (H_{11}(\bm t_0)+\epsilon^\star)^{d/(4+d)}(1+\epsilon)\Lambda^{(\ell)}M^{-2d/(d+4)}d^{d/(d+4)} \alpha_d^{-\frac{2d}{d+4}} \qquad \quad (\mbox{by }\eqref{convex_c}) \nonumber \\
		&=& (1+\epsilon)\Lambda^{(\ell)}\rho_n  \left( \Pi_{i=1}^d (H_{ii}(\bm t_0)+\epsilon^\star) \right)^{1/(d+4)}d^{d/(d+4)} \alpha_d^{-\frac{2d}{d+4}} \nonumber \qquad \quad (\mbox{by }\eqref{Convex_h_M})\\
		&\leq & (1+\epsilon)\Lambda^{(\ell)}\rho_n (1+\epsilon)^{d/(d+4)} L_{2,\star}[f,\bm t_0]^{d/(d+4)} d^{d/(d+4)} \alpha_d^{-\frac{2d}{d+4}} \nonumber \quad \qquad (\mbox{by }\eqref{HessianUB2}) \\
		&=&(1+\epsilon)^{(2d+4)/(d+4)}\Delta^{(\ell \star)} \rho_n L_{2,\star}[f,\bm t_0]^{d/(d+4)}\label{optconst49}
	\end{eqnarray}
	As $\epsilon$ is arbitrary, our assertion is proved by using~\eqref{optconst39} and~\eqref{optconst49}. For bounding $(\hat{u}-f)(\bm t_0)$, one again works with the transformed kernel $\xi_2^u(\bm x) := \psi_2^u(\bm x/\sqrt{2})$
	to get a result analogous to Lemma \ref{techresl}.
	This gives rise to the $(\sqrt{2})^d$ term in $\Delta^{(u)}$ and completes the proof of part (b) of Theorem \ref{opconstantcx}.
	\newline
	
	\noindent (c) The proof of part (c) is very similar to that of part (b), so we highlight the differences. To begin with, fix $\varepsilon>0$. Now, as $H({\bm t}_0)$ is a diagonal matrix and $f$ is twice-continuously differentiable, we can find a hyperrectangle $B_\infty(\bm t_0,\bm h_0)$ small enough, such that
	\begin{equation}\label{off-d}
		\sup_{\bm x \in B_\infty(\bm t_0,\bm h_0)} |H_{ij}(\bm x)| \le \frac{\varepsilon \left(\min_{1\le i\le d} H_{ii}(\bm t_0) + \epsilon^\star\right)}{d}~ \qquad \quad \mbox{ for all }\;\;i\ne j.
	\end{equation}
	As before, for the function $g(\bm x) := (f(\bm t+ \bm h\star \bm x) - f(\bm t))/(H_{vv}(\bm t_0) + \epsilon^\star) h_v^2$ (cf.~\eqref{scale7})
	note that:
	\begin{eqnarray*}
		&&\sum_{j=1}^d\left|\frac{\partial g(\bm y)}{\partial x_j}- \frac{\partial g(\bm z)}{\partial x_j}\right|\nonumber\\ &\le& \sum_{j=1}^d \sum_{i=1}^d    \left|\frac{\partial}{\partial x_j} g(y_1,\ldots,y_{i-1},z_i,\ldots,z_d) -  \frac{\partial}{\partial x_j} g(y_1,\ldots,y_{i},z_{i+1},\ldots,z_d)\right| 
		\\ &\le& \sum_{j=1}^d |y_j-z_j|\sup_{\bm u \in L(\bm y,\bm z)}\left|\frac{\partial^2}{\partial x_j^2} g(\bm u)\right| +  \sum_{j=1}^d \sum_{i\ne j} |y_i-z_i| \sup_{\bm u \in L(\bm y,\bm z)}\max_{1\le i\le d}\left|\frac{\partial^2}{\partial x_i \partial x_j} g(\bm u)\right|\nonumber\\
		&\le&  \|\bm y-\bm z\|_1\sup_{\bm u \in L(\bm y,\bm z)}\max_{1\le i\le d}\left|\frac{\partial^2}{\partial x_i^2 } g(\bm u)\right| + \sum_{i=1}^d \sum_{j\ne i} |y_i-z_i| \sup_{\bm u \in L(\bm y,\bm z)}\max_{1\le p\ne q\le d}\left|\frac{\partial^2}{\partial x_p \partial x_q} g(\bm u)\right|\\ 
		&\le& \|\bm y-\bm z\|_1\left(\sup_{\bm u \in L(\bm y,\bm z)}\max_{1\le i\le d}\left|\frac{\partial^2}{\partial x_i^2 } g(\bm u)\right| + \varepsilon\right)\nonumber
	\end{eqnarray*}
	where $L(\bm y,\bm z)$ denotes the hyperrectangle defined by the two extreme points $\bm y$ and $\bm z$. The last inequality in the above display follows from noticing that $$\left|\frac{\partial^2}{\partial x_p \partial x_q} g(\bm u)\right| =\frac{\left|\frac{\partial^2}{\partial x_p \partial x_q } f(\bm t + \bm h \star \bm u)\right||h_p h_q|}{(H_{vv}(\bm t_0) + \epsilon^\star)h_v^2} \leq \frac{\varepsilon \left(\min_{1\le i\le d} H_{ii}(\bm t_0) + \epsilon^\star\right)}{d \sqrt{ (H_{pp}(\bm t_0) + \epsilon^\star)(H_{qq}(\bm t_0) + \epsilon^\star)}} \leq \frac{\varepsilon}{d},$$
	where the second to last inequality follows from~\eqref{eq:h_i-cvx} and \eqref{off-d}. The rest of the proof is exactly similar to the proof of part (b), modulo the factor $d^{-4/(d+4)}$ missing in the expression for $c$.
	
	
	\section*{Acknowledgements}
	The authors would like to thank Lutz D\"{u}mbgen for helpful discussions.

	\bibliographystyle{apalike}
	\bibliography{biblio}

	\appendix

	\section{Some auxiliary results}\label{techres}
	{\bf Proof of Lemma~\ref{techresl}}:
	To begin with, note that $\psi_2^\ell(\bm x) = (G_A - G_0)\left(\sqrt{2(d+3)/(d+1)} \bm x\right)$, where
	$$G_A(\bm y) := \frac{\|\bm y\|^2}{2} \mathbbm{1}_{\|\bm y\| \le \sqrt{2(d+3)/(d+1)}}\quad \text{and}\quad G_0(\bm y) := \left(-1+ \frac{\sqrt{2}(d+2)}{\sqrt{(d+1)(d+3)}} \|\bm y\|\right) \mathbbm{1}_{\|\bm y\| \le \sqrt{2(d+3)/(d+1)}}~.$$
	We will now prove the following claim:
	
	\begin{claim}\label{h21mns}
		For all $g \in \mathbb{H}_{2,1}'$ with support $\mathscr{B}_d:= \{\bm y: \|\bm y\| \le \sqrt{2(d+3)/(d+1)}\}$, $G_A-g$ is convex on the set $\mathscr{B}_d$.
	\end{claim}
	\noindent For proving Claim \ref{h21mns}, it suffices to show that for every $\bm v \in \mathbb{R}^d$ such that $\sum_{i=1}^d v_i \ge 0$, the function $f_{\bm v}:\mathbb{R}\mapsto \mathbb{R}$ defined as
	$f_{\bm v}(\alpha) := \frac{\|\alpha \bm v\|^2}{2} - g(\alpha \bm v)$ is convex. Towards this, note that:
	$$f_{\bm v}'(\alpha) = \alpha \|\bm v\|^2 - \bm v^\top \nabla g(\alpha \bm v)~.$$
	Now, take any pair $(\alpha, \beta)$ such that $\alpha < \beta$, and note that:
	\begin{eqnarray*}
		\left|\bm v^\top \nabla g(\alpha \bm v) - \bm v^\top \nabla g(\beta \bm v)\right| &\le& \|\bm v\| \|\nabla g(\alpha \bm v) - \nabla g(\beta \bm v)\|\\&\le& \|\bm v\| \sum_{i=1}^d |\nabla_i g(\alpha \bm v) - \nabla_i  g(\beta \bm v)|\\&\le& \|\bm  v\|\|(\alpha-\beta)\bm v\| = (\beta-\alpha)\|\bm v\|^2~.  
	\end{eqnarray*}
	The last inequality followed from the fact that $g \in \mathbb{H}_{2,1}'$. Hence, we have:
	$$\bm v^\top \nabla g(\beta \bm v) - \bm v^\top \nabla g(\alpha \bm v) \le \beta \|\bm v\|^2 - \alpha \|\bm v\|^2~\implies~f_{\bm v}'(\alpha) \le f_{\bm v}'(\beta)~,$$
	thereby showing that $f_{\bm v}$ is convex, and completing the proof of Claim \ref{h21mns}.
	
	With Claim \ref{h21mns} in hand, we are now ready to prove Lemma \ref{techresl}. Defining $\psi := G_A-G_0$, we have in view of Claim \ref{h21mns} and \eqref{short}, that for any $g \in \mathbb{H}_{2,1}'$ with support $\mathscr{B}_d$,
	
	$$\langle G_A-g, \psi\rangle \le (G_A-g)(\boldsymbol{0})\langle 1,\psi\rangle$$
	and hence, we have:
	\begin{eqnarray*}
		\langle g, \psi\rangle &=& \langle G_A,\psi\rangle - \langle G_A - g, \psi \rangle\\&\ge & \langle G_A,\psi\rangle - (G_A-g)(\boldsymbol{0})\langle 1,\psi\rangle\\&=& \langle G_A,\psi\rangle + g(\boldsymbol{0})\langle 1,\psi\rangle\\&=& \|\psi\|^2 + \langle G_0,\psi\rangle + g(\boldsymbol{0})\langle 1,\psi\rangle\\&=& \|\psi\|^2 + (g(\boldsymbol{0})-1)\langle 1,\psi\rangle
	\end{eqnarray*}
	where the last equality followed from the fact that $\langle G_0+1,\psi\rangle = 0$, which follows by an argument similar to the proof of \eqref{orthbreak}. 
	Since $g(\boldsymbol{0}) \ge 0$, we conclude that:
	$$
	\langle g, \psi\rangle \ge \|\psi\|^2 - \langle 1, \psi\rangle
	$$
	which proves the following claim:
	
	\begin{claim}\label{h21mn8s}
		For all $g \in \mathbb{H}_{2,1}'$ with support $\mathscr{B}_d:= \{\bm y: \|\bm y\| \le \sqrt{2(d+3)/(d+1)}\}$, and $\psi: \mathscr{B}_d\to \R$ defined as:
		$\psi(\bm y) := \psi_2^\ell\left(\sqrt{\frac{d+1}{2(d+3)}}\bm y\right)$, we have:
		$$
		\langle g, \psi\rangle \ge \|\psi\|^2 - \langle 1, \psi\rangle
		$$
	\end{claim}
	
	Lemma \ref{techresl} now follows from Claim \ref{h21mn8s} on observing that $\psi_2^\ell(\bm y) = \psi\left(\sqrt{2(d+3)/(d+1)} \bm y\right)$.
	\qed

	\begin{lem}\label{nndrs}
		If $A := ((A_{i,j}))_{1\le i,j\le n}$ is an $n\times n$ nonnegative definite matrix, then there exists $1\le i\le n$ such that $A_{i,i} =\max_{1\le j,k\le n} A_{j,k}$.    
	\end{lem}
	\begin{proof}
		Suppose that there exists $1\le i<j\le n$ such that $A_{i,j}>\max_{1\le i\le n} A_{i,i}$. Note that if $\bm e_i$ denotes the vector with the $i^{\mathrm{th}}$ entry $1$ and all other entries $0$, then:
		$$0\le (\bm e_i - \bm e_j)^\top A \bm (\bm e_i - \bm e_j) = A_{i,i} + A_{j,j}-2A_{i,j}~\implies~A_{i,i}+A_{j,j} \ge 2 A_{i,j}~,$$ a contradiction! This proves Lemma \ref{nndrs}.
	\end{proof}

\end{document}